\documentclass[reqno]{amsart}

\usepackage{amssymb}

\usepackage{xcolor}
\definecolor{myred}{HTML}{D70000}
\definecolor{mygreen}{HTML}{007040}
\usepackage[colorlinks=true,linkcolor=myred,citecolor=mygreen,linktocpage=true]{hyperref}
\usepackage{mathtools}

\usepackage{mathdots}

\newtheorem{theorem}{Theorem}
\newtheorem{lemma}[theorem]{Lemma}
\newtheorem{corollary}[theorem]{Corollary}
\theoremstyle{remark}
\newtheorem{remark}{Remark}
\theoremstyle{definition}
\newtheorem{definition}{Definition}
\newcommand{\pmat}{\setlength{\arraycolsep}{0.3em}\begin{pmatrix}}
\newcommand{\epmat}{\end{pmatrix}}
\numberwithin{theorem}{section}
\numberwithin{equation}{section}

\renewcommand{\le}{\leqslant}
\renewcommand{\ge}{\geqslant}

\renewcommand{\Re}{\operatorname{Re}}
\renewcommand{\Im}{\operatorname{Im}}

\newcommand{\sign}{\operatorname{sign}}
\newcommand{\Arg}{\operatorname{Arg}}
\DeclareRobustCommand{\SkipTocEntry}[5]{}

\title[Total nonnegativity of infinite Hurwitz matrices]
{Total nonnegativity of infinite Hurwitz matrices of entire and meromorphic
    functions} 

\author{Alexander Dyachenko}
\address{TU-Berlin, Sekretariat MA 3-6, Stra\ss e des 17. Juni 136,
   D-10623 Berlin, Germany}
\email{dyachenk@math.tu-berlin.de}
\thanks{This work
    was financially supported by the European Research Council under the European Union’s
    Seventh Framework Programme (FP7/2007--2013)/ERC grant agreement no. 259173.}
\def\noTOC#1{}
\begin{document}
\begin{abstract}
    In this paper we fully describe functions generating the infinite totally nonnegative
    Hurwitz matrices. In particular, we generalize the well-known result by Asner and
    Kemperman on the total nonnegativity of the Hurwitz matrices of real stable polynomials. An
    alternative criterion for entire functions to generate a P\'olya frequency sequence is also
    obtained. The results are based on a connection between a factorization of totally
    nonnegative matrices of the Hurwitz type and the expansion of Stieltjes meromorphic
    functions into Stieltjes continued fractions (regular $C$-fractions with positive
    coefficients).
\end{abstract}
\maketitle

\tableofcontents
\section{Introduction}
Functions mapping the upper half-plane of the complex plane into itself
($\mathcal{R}$-functions) are well studied and play a significant role in applications. The
subclass~$\mathcal{S}$ of $\mathcal{R}$-functions, the functions that are regular and
nonnegative over the nonnegative semi-axis (also known as Stieltjes functions) is of particular
interest. In this paper we demonstrate a connection of meromorphic $\mathcal{S}$-functions with
total nonnegativity of corresponding Hurwitz-type matrices (Theorem~\ref{th:main1}). As an
application, we study the following problem on the distribution of zeros.

A polynomial with no roots with a positive real part is called \emph{quasi-stable}. Asner
(see~\cite{Asner}) established that the Hurwitz matrix of a real quasi-stable polynomial is
totally nonnegative (although there are polynomials with totally nonnegative Hurwitz matrices
which are not quasi-stable). A matrix is called \emph{totally nonnegative} if all of its minors
are nonnegative. In addition, Kemperman (see~\cite{Kemperman}) showed that quasi-stable
polynomials have totally nonnegative \emph{infinite} Hurwitz matrices.

It turns out that the replacement of finite Hurwitz matrices with infinite Hurwitz matrices
allows us to prove the converse:
\emph{a polynomial is quasi-stable if its infinite Hurwitz matrix is totally nonnegative}. The
key to this is given in~\cite{Holtz}: a special matrix factorization, which was successfully
applied to a closely related problem in~\cite{HoltzTyaglov}. Moreover, when a theorem involves
an infinite Hurwitz matrix, it is natural to suggest that it can be generalized to entire
functions or power series. The first goal of the present paper is to obtain the following
extension of the results from~\cite{Asner},~\cite{Kemperman} and~\cite{Holtz} to power series,
including the converse result.

\begin{theorem}\label{th:qstable_tnn}
    Given a power series $f(z)=z^j\sum_{k=0}^\infty\,f_kz^k$ in the complex variable~$z$, where
    $f_0>0$ and $j$ is a nonnegative integer, the infinite Hurwitz matrix
    \begin{equation*}
        \mathcal H_f={\begin{pmatrix}
                f_0 & f_2 & f_4 & f_6 & f_8 & \hdots\\
                0 & f_1 & f_3 & f_5 & f_7 & \hdots\\
                0 & f_0 & f_2 & f_4 & f_6 & \hdots\\
                0 &   0 & f_1 & f_3 & f_5 & \hdots\\
                0 &   0 & f_0 & f_2 & f_4 & \hdots\\
                \vdots & \vdots & \vdots & \vdots & \vdots & \ddots
            \end{pmatrix}}
    \end{equation*}
    is totally nonnegative if and only if the series~$f$ converges to a function of the form
    \begin{equation}\label{eq:f_as_quo}
        f(z)=Cz^je^{\gamma_1 z+\gamma_2 z^2}
        \frac{
            \prod_{\mu}\left( 1+\frac {z}{x_\mu}\right)\prod_{\nu}\left(
                1+\frac {z}{\alpha_\nu}\right)\left(1+\frac {z}{\overline\alpha_\nu}\right)}{
            \prod_{\lambda}\left( 1+\frac {z}{y_\lambda}\right)
            \left( 1-\frac {z}{y_\lambda}\right)},
    \end{equation}
    where $C, \gamma_1, \gamma_2 \ge 0$, $x_\mu,y_\lambda > 0$, $\Re\alpha_\nu\ge 0$,
    $\Im\alpha_\nu>0$ and
    \[\textstyle
        \sum_{\mu} \frac 1{x_\mu}
        + \sum_{\nu} \Re(\frac{1}{\alpha_\nu})
        + \sum_{\nu} \frac 1{|\alpha_\nu|^2}
        + \sum_{\lambda} \frac 1{y_\lambda^2}<\infty.
    \]    
\end{theorem}
\begin{remark}
    Stating herein that a power series converges, by default we assume it to be convergent in a
    neighbourhood of the origin. Moreover, where it creates no uncertainties we use the same
    abbreviation for the series and a function it converges to.
\end{remark}
\begin{remark}
    It is possible that $\{x_\mu\}_\mu\cap\{y_\lambda\}_\lambda\ne\varnothing$ in the
    expression~\eqref{eq:f_as_quo}. If so, the coinciding negative zeros and poles of the
    function~$f(z)$ cancel each other out, while its positive poles remain untouched. For
    example, although the series $\sum_{k=0}^\infty z^k$ satisfies Theorem~\ref{th:qstable_tnn},
    it converges to the function~$\frac{1}{1-z}$ with a unique positive pole. The number of such
    cancellations may be infinite, however it cannot affect the convergence of involved infinite
    products.
\end{remark}
Our second goal is achieved by Theorem~\ref{th:negroots_p0dec}, which is an extension
of~\cite[Theorem 4.29]{HoltzTyaglov}.

\begin{theorem}\label{th:negroots_p0dec}
    A power series~$f(z)={\sum_{k=0}^\infty}\, f_kz^k$ with $f_0> 0$ converges to an entire function of
    the form
    \begin{equation}\label{eq:ent_tnns_generator}
        f(z)=f_0\,e^{\gamma z}\prod_{\nu} \left(1+\frac {z}{\alpha_\nu}\right),
    \end{equation}
    where $\gamma\ge 0$, $\alpha_\nu>0$ for all~$\nu$ and $\sum_{\nu}\frac{1}{\alpha_\nu} <\infty$,
    if and only if the infinite matrix
    \[
    \mathcal D_f={\begin{pmatrix}
            f_0 & f_1 & f_2 & f_3 & f_4 & \hdots\\
            0 & f_1 & 2f_2 & 3f_3 & 4f_4 & \hdots\\
            0 & f_0 & f_1 & f_2 & f_3 & \hdots\\
            0 &     0 & f_1 & 2f_2& 3f_3 & \hdots\\
            0 &     0 & f_0 & f_1 & f_2 & \hdots\\
            \vdots & \vdots & \vdots & \vdots & \vdots & \ddots
        \end{pmatrix}}
    \]
    is totally nonnegative.
\end{theorem}
This theorem complements the following well-known criterion established by Aissen, Edrei,
Schoenberg and Whitney.
\begin{theorem}[\cite{AESW0,AESW1,AESW2}, see also~{\cite[Section~8 \S 5]{Karlin}}]\label{th:AESW}
    Given a formal power series~$f(z)={\sum_{k=0}^\infty}\, f_kz^k$, $f_0> 0$, the Toeplitz
    matrix
    \begin{equation}\label{eq:T_form}
        T(f)={\begin{pmatrix}
                f_0 & f_1 & f_2 & f_3 & f_4 & \hdots\\
                0 & f_0 & f_1 & f_2 & f_3 & \hdots\\
                0 &   0 & f_0 & f_1 & f_2 & \hdots\\
                0 &   0 & 0   & f_0 & f_1 & \hdots\\
                0 &   0 &   0 & 0   & f_0 & \hdots\\
                \vdots & \vdots & \vdots & \vdots & \vdots & \ddots
            \end{pmatrix}}
    \end{equation}
    is totally nonnegative if and only if $f$ converges to a meromorphic function of the form:
    \begin{equation}\label{eq:tnns_generator}
        f(z)=f_0\,e^{\gamma z}\frac{\prod_{\nu} \left(1+\frac {z}{\alpha_\nu}\right)}
        {\prod_{\mu} \left(1-\frac {z}{\beta_\mu}\right)},
    \end{equation}
    where
    $\gamma\ge 0$, $\alpha_\nu,\beta_\mu>0$ for all $\mu,\nu$ and
    $\sum_{\nu}\frac{1}{\alpha_\nu}+\sum_{\mu}\frac{1}{\beta_\mu} <\infty$.
\end{theorem}

If we require the series $f(z)$ to represent an entire function under the assumptions of
Theorem~\ref{th:AESW}, we obtain that it has the form~\eqref{eq:ent_tnns_generator}. We prove
Theorems~\ref{th:qstable_tnn} and~\ref{th:negroots_p0dec} in
Section~\ref{sec:appl-entire-funct}.

A sequence $(f_k)_{k=0}^\infty$ is commonly called \emph{totally positive} (\emph{e.g.} \cite{AESW1}),
or a \emph{P\'olya frequency sequence} (\emph{e.g.} \cite{Karlin}), whenever the matrix $T(f)$ defined
by \eqref{eq:T_form} is totally nonnegative. By Theorem~\ref{th:AESW}, the general form of its
generating function is given by the formula~\eqref{eq:tnns_generator}.

\begin{definition}\label{def:R_R-1}
    We denote by $\mathcal R$ ($\mathcal R^{-1}$ resp.) the class of all
    meromorphic\footnote{In general, the condition to be meromorphic is replaced by less
        restrictive~$F(\overline z)=\overline{F(z)}$. Basic properties of $\mathcal R$-functions
        can be found, for example, in~\cite{KreinKac} and (for the meromorphic case)
        in~\cite{Wigner}. For brevity's sake we confine ourselves to meromorphic functions
        only.} functions $F(z)$ analytic in the complement of the real axis and such that
\[
\frac{\Im{F(z)}}{\Im{z}}\ge 0
\qquad\bigg(\text{or }
\frac{\Im{F(z)}}{\Im{z}}\le 0\text{ for }F\in\mathcal{R}^{-1}\text{ resp}\bigg).
\]
\end{definition}
Note that it is a straightforward consequence of the definition that $\mathcal{R}$- and
$\mathcal{R}^{-1}$-functions are real (\emph{i.e.}~map the real line into itself). Furthermore,
our definition includes real constants (like in~\cite{KreinKac}) into both classes $\mathcal{R}$
and $\mathcal{R}^{-1}$ although sometimes they are excluded in the literature
(\emph{e.g.}~\cite{Wigner}).

\begin{definition}
    Denote by $\mathcal{S}$ the subclass of $\mathcal{R}$-functions that are regular and
    nonnegative over the nonnegative reals. (Since $\mathcal{S}$-functions are meromorphic, they
    can have only negative poles and nonpositive zeros.)
\end{definition}
Consider the {\em infinite Hurwitz-type matrix} (\emph{i.e.} the
{\em matrix of the Hurwitz type})
\begin{equation}\label{eq:Hpq_def}
    H(p,q)={\begin{pmatrix}
        b_0& b_1& b_2 & b_3 & b_4 & b_5 & \hdots\\
        0 & a_0 & a_1 & a_2 & a_3 & a_4 & \hdots\\
        0 & b_0 & b_1 & b_2 & b_3 & b_4 & \hdots\\
        0 &   0 & a_0 & a_1 & a_2 & a_3 & \hdots\\
        0 &   0 & b_0 & b_1 & b_2 & b_3 & \hdots\\
        \vdots & \vdots & \vdots & \vdots & \vdots & \vdots & \ddots
    \end{pmatrix}},
\end{equation}
where $p(z)=\sum_{k=0}^\infty\,a_kz^k$ and $q(z)=\sum_{k=0}^\infty\,b_kz^k$ are formal power
series. Given two arbitrary constants $c$ and $\beta$, we also consider the matrix
\begin{equation}\label{eq:J}
J(c,\beta)={\begin{pmatrix}
    c & \beta & 0 & 0 & 0 & 0 & \hdots\\
    0 & 0 & 1 & 0 & 0 & 0 & \hdots\\
    0 & 0 & c & \beta & 0 & 0 & \hdots\\
    0 & 0 & 0 & 0 & 1 & 0 & \hdots\\
    0 & 0 & 0 & 0 & c & \beta & \hdots\\
    \vdots & \vdots & \vdots & \vdots & \vdots & \vdots & \ddots
\end{pmatrix}}.
\end{equation}
Matrices of this type will appear in our factorizations below.

Finally, for an infinite matrix $A=(a_{ij})_{i,j=1}^\infty$ and a fixed number $\rho$,
$0<\rho\le 1$, we consider the matrix norm
\[
\|A\|_\rho \coloneqq \sup_{i\ge 1}\sum_{j=1}^\infty \rho^{j-1} |a_{ij}|.
\]

\begin{remark}
    Convergence in this norm implies entry-wise convergence. Moreover, the norm $\|A\|_\rho$ of
    a matrix~$A$ coincides with the norm of the operator
    \[A_\rho:x\mapsto A\cdot\operatorname{diag}(1,\rho,\rho^2,\dots)\cdot x,\]
    acting on the space~$l_\infty$ of bounded sequences.
\end{remark}
\begin{remark}
    Let functions $g(z)$, $p(z)$, $q(z)$ and $g^{(k)}(z)$, $p^{(k)}(z)$, $q^{(k)}(z)$,
    $k=1,2,\dots$, be holomorphic on $\overline{D}_\rho\coloneqq\{z\in\mathbb{C}:|z|\le\rho\}$.
    Then the condition \[\lim_{k\to\infty}\|T(g^{(k)})-T(g)\|_\rho= 0 \]
    is equivalent to the uniform convergence of $g^{(k)}(z)$ to $g(z)$ on
    $\overline{D}_\rho$, and the condition
    \[\lim_{k\to\infty}\|H(p^{(k)},q^{(k)})-H(p,q)\|_\rho= 0 \]
    is equivalent to the uniform convergence of $p^{(k)}(z)$ to $p(z)$ and $q^{(k)}(z)$ to $q(z)$
    on~$\overline{D}_\rho$.
\end{remark}

Now we can formulate the more important result of this paper concerning properties of
$\mathcal{S}$-functions. It has its own value apart from the proofs of
Theorems~\ref{th:qstable_tnn} and~\ref{th:negroots_p0dec}.

\begin{theorem}\label{th:main1}
    Consider the ratio~$F(z)=\frac{q(z)}{p(z)}$ of power
    series~$p(z)=\sum_{k=0}^\infty\,a_kz^k$ and $q(z)=\sum_{k=0}^\infty\,b_kz^k$, normalized by
    the equality~$p(0)=a_0=1$. The following conditions are equivalent:
    \begin{enumerate}
    \item The infinite Hurwitz-type matrix $H(p,q)$ defined by~\eqref{eq:Hpq_def} is totally
        nonnegative.
    \item \label{item:2f}
        The matrix $H(p,q)$ possesses the infinite factorization
        \begin{equation}\label{eq:H_inf_fact}
            H(p,q)= \lim_{j\to\infty} \big( J(b_0,\beta_0)\,J(1,\beta_1) \cdots
            J(1,\beta_j)\big) \,H(1,1)\,T(g)
        \end{equation}
        converging in~$\|\cdot\|_\rho$-norm for some $\rho$, $0<\rho\le 1$. Here $b_0\ge 0$ and
        the sequence $(\beta_j)_{j\ge{0}}$ is nonnegative, has a finite sum and contains no
        zeros followed by a nonzero entry, that is
        \begin{equation}\label{eq:H_inf_fact_beta}
            \begin{gathered}
                \beta_0,\beta_1,\dots,\beta_{\omega-1}>0, \quad
                \beta_{\omega}=\beta_{\omega+1}=\dots=0,\\
                0\le\omega\le\infty, \quad \text{and} \quad
                \sum_{j=0}^\infty\beta_j<\infty.
            \end{gathered}
        \end{equation}
        The matrix~$T(g)$ denotes a totally nonnegative Toeplitz matrix of the
        form~\eqref{eq:T_form} with ones on its main diagonal.
    \item \label{item:3f}
        The ratio~$F(z)$ is a meromorphic $\mathcal S$-function; its numerator $q(z)$ and
        denominator $p(z)$ are entire functions of genus~$0$ up to a common meromorphic
        factor~$g(z)$ of the form~\eqref{eq:tnns_generator}, $g(0)=1$.
    \end{enumerate}
\end{theorem}

\begin{remark}
    Note that
    \begin{equation}\label{eq:J(c,0)H(1,1)}
        J(c,0)H(1,1) = H(1,c).
    \end{equation}
    If $\omega$ is a finite number in~\eqref{eq:H_inf_fact_beta}, then
    $\beta_{\omega+1}=\beta_{\omega+2}=\dots=0$ implying
    \[J(b_0,\beta_0) \cdots J(1,\beta_{\omega+1})\,H(1,1)
    =J(b_0,\beta_0) \cdots  J(1,\beta_{\omega+1})\,J(1,\beta_{\omega+2})\,H(1,1)=\cdots.\]
    As a consequence, the factorization~\eqref{eq:H_inf_fact} can be expressed as follows in
    this case
    \begin{equation}\label{eq:H_fin_fact}
        H(p,q)=\begin{cases}
            J(b_0,0)\,H(1,1)\,T(g) = H(1,b_0)\,T(g) &\text{if }\omega=0;\\
            J(b_0,\beta_0)\,J(1,\beta_1) \cdots J(1,\beta_{\omega-1})\,H(1,1)\,T(g)
            &\text{if }0<\omega<\infty.
        \end{cases}
    \end{equation}
\end{remark}

\begin{remark}
    The number $\rho$ in Theorem~\ref{th:main1} can be anywhere in $(0,1]\cap(0,\rho_0)$, here
    $\rho_0$ denotes the radius of convergence of~$g(z)$ (which is positive by
    Theorem~\ref{th:AESW}).

    The matrix $T(g)$ from the condition~\ref{item:2f} of Theorem~\ref{th:main1} is the Toeplitz
    matrix of the function $g(z)$ from \ref{item:3f} given by~\eqref{eq:Hpq_def}.

    If we require~$p(z)$ and $q(z)$ to be entire functions in Theorem~\ref{th:main1}, then the
    function~$g(z)$ has the form~\eqref{eq:ent_tnns_generator} or $g(z)\equiv 1$ and
    \eqref{eq:H_inf_fact} converges in $\|\cdot\|_1$.
\end{remark}
\begin{remark}
    In the case $q(0)=b_0=0$ it can be convenient to ``trim'' the matrix~$H(p,q)$ by removing
    its first row and its trivial first column. This corresponds to replacing~$J(0,\beta_0)$ in
    the factorization~\eqref{eq:H_inf_fact2} by its diagonal analogue
    $\operatorname{diag}(1,\beta_0,1,\beta_0,\dots)$.
\end{remark}
\begin{remark}
    Since entire functions of genus $0$ have unique Weierstra\ss' representation, it makes sense
    to consider the \emph{greater common divisor} of a subset of this class. Accordingly, two
    entire functions~$p$ and~$q$ of genus~$0$ are \emph{coprime} whenever $\gcd(p,q)\equiv 1$.
\end{remark}

Consider the continued fraction
\begin{equation}\label{eq:FasCF2a}
    b_0+\frac{\beta_0 z}1 \,\underset{+}{}\,
    \frac{\beta_1 z}1 \,\underset{+}{}\,
    \frac{\beta_2 z}1 \,\underset{+\cdots+}{}\,
    \frac{\beta_{\omega-1} z}1,\ \ 
    b_0\ge 0,\ \ \beta_0,\beta_1,\dots,\beta_{\omega-1}>0,\ \  0\le\omega\le\infty,
\end{equation}
where we combine both finite (terminating) and infinite cases. If the continued fraction is
infinite, we assume $\omega=\infty$. The following Corollary (see its proof in
Subsection~\ref{sec:S-functions-as-CF}) allows us to connect the
factorization~\eqref{eq:H_inf_fact} with continued fractions of this type.
\begin{corollary}\label{cr:main2}
    Let~$F(z)=\frac{q(z)}{p(z)}$ be a meromorphic $\mathcal{S}$-function, where the entire
    functions $p(z)$ and $q(z)$ are of genus~$0$. Then it can be expanded into a uniformly
    convergent continued fraction of the form~\eqref{eq:FasCF2a} with exactly the same
    coefficients $b_0$ and $(\beta_j)_{j=0}^{\omega-1}$, $\sum_{j=0}^{\omega-1}\beta_j<\infty$,
    as in the factorization~\eqref{eq:H_inf_fact2} of the matrix~$H(p,q)$. No other continued
    fractions of the form
    \[
    F(z) = c_0+\frac{c_1 z^{r_1}}1 \,\underset{+}{}\,
    \frac{c_2 z^{r_2}}1 \,\underset{+}{}\,
    \frac{c_3 z^{r_3}}1 \,\underset{+\cdots+}{}\,
    \frac{c_\omega z^{r_\omega}}1,
    \]
    where $c_j\ne 0$ and $r_j\in\mathbb{N}$ for $j=1,\dots,\omega$, $0\le \omega\le \infty$, can
    correspond to the Taylor series of~$F(z)$.
\end{corollary}
\begin{remark}
    Corollary~\ref{cr:main2} implies that
    \emph{each pair~$(p(z), q(z))$ satisfying Theorem~\ref{th:main1} determine a unique
        factorization of the form~\eqref{eq:H_inf_fact}}.
\end{remark}

Let $p(z)$ and $q(z)$ be real polynomials. Denote
\[
u(z)\coloneqq\sum_{k=0}^na_kz^{n-k}=z^np\left(\frac{1}{z}\right)\ \text{and}\
v(z)\coloneqq\sum_{k=0}^nb_kz^{n-k}=z^nq\left(\frac{1}{z}\right),
\]
where $n=\max\{\deg p,\deg q\}$. In this case it is more common to work with the matrix
$\tilde{H}(u,v)\coloneqq H(p,q)$ instead of~$H(u,v)$.

In fact, Theorem~\ref{th:main1} extends the following result by Holtz and Tyaglov to
meromorphic functions. In~\cite[Theorems~1.46 and~3.43, Corollaries~3.41--3.42]{HoltzTyaglov}
they established that the matrix~$\tilde{H}(u,v)$ is totally nonnegative if and only if it can
be factored as follows
\begin{equation}\label{eq:tH_fact}
\tilde{H}(u,v)=J(c_0,1)\dots J(c_j,1)\,H(1,0)\,T(g),\quad c_1,\dots,c_j>0,
\end{equation}
where $T(g)$ is totally nonnegative and $g=\gcd(u,v)$. Note that the
factorization~\eqref{eq:tH_fact} corresponds to~\eqref{eq:H_fin_fact} after the substitutions
$b_0=c_0$, $\beta_0=(c_1)^{-1}$ and $\beta_{i-1}=(c_{i-1}c_{i})^{-1}$ for $i=2,\dots,j$. Moreover, by
Theorem~3.44 from~\cite{HoltzTyaglov} the matrix~$\tilde{H}(u,v)$ is totally nonnegative if and
only if $v(z)$ and $u(z)$ have no positive zeros and $\frac{v}{u}\in\mathcal{R}^{-1}$. Since
\[
\frac{p(z)}{q(z)} =\frac{v\left(\frac 1z\right)}{u\left(\frac 1z\right)},
\]
we obtain the polynomial analogue of Theorem~\ref{th:main1}.

Earlier, Holtz (see~\cite{Holtz}) found that the infinite Hurwitz matrix of a stable polynomial
(\emph{i.e.} a polynomial with no roots with nonnegative real part) has the
factorization~\eqref{eq:tH_fact} with $T(g)$ equal to the identity matrix. Additionally, each of
the factors~$J(c_j,1)$ corresponds to a step of the Routh scheme. These factorizations coincide
because the problems considered in~\cite{Holtz} and~\cite{HoltzTyaglov} are closely connected
(see, for example, the monographs of Gantmakher~\cite[Ch.~XV]{Gantmakher} and
Wall~\cite[Chapters IX and~X]{Wall}). In order to deduce Theorem~\ref{th:qstable_tnn} from
Theorem~\ref{th:main1}, we are using the same underlying connection.

\section{Basic facts}\label{sec:basic_facts}
Here we consider some facts that are quite significant, although, in fact, they are not new. We
put them here to introduce the area and our notation. The most ``non-standard'' assertion here
is Lemma~\ref{lemma:decomp_decrease}, since it reverses the approach of
Theorem~\ref{th:Stieltjes}.
\subsection{$\mathcal S$-functions in terms of Hurwitz-type matrices}
\label{sec:S-functions-vs-Hm}
Consider power series
\begin{equation}\label{eq:pq_exp}
    p(z)=\sum_{k=0}^\infty\,a_kz^k,\quad a_0=1,
    \quad\text{and}\quad
    q(z)=\sum_{k=0}^\infty\,b_kz^k,\quad b_0\ge 0.
\end{equation}
Let us introduce the following notations
\[
p_0(z)\coloneqq p(z),\quad p_{-1}(z)\coloneqq q(z),
\quad H\coloneqq H(p,q),\quad\text{and}\quad H_0\coloneqq H(p_0,p_{-1})=H.
\]
Denote the minor of a matrix~$A$ with rows $i_1,i_2,\dots,i_k$ and columns $j_1,j_2,\dots,j_k$
by
\[A{\pmat i_1&i_2&\dots& i_k\\j_1&j_2&\dots& j_k \epmat}.\]
In addition set
\begin{equation}\label{eq:H_minor_sm}
    A^{(k)}\coloneqq A{\pmat 2&3&\dots& k\\2&3&\dots& k \epmat}.
\end{equation}

If~the number~$\beta_0=b_1-b_0\,a_1 = H_0^{(3)}$ is nonzero, we define
\[p_1(z)\coloneqq\dfrac{q(z)-b_0\,p(z)}{\beta_0 z}\text{\quad and\quad}H_1\coloneqq H(p_1,p_0).\]
Now we can perform the same manipulations with the pair~$p_1(z)$, $p_0(z)$. That is, we can make
the next step of the following algorithm.

At the $j$th step, $j=0,1,2\dots$, the series~$p_{j}(z)$ and~$p_{j-1}(z)$ are already defined,
as well as the matrix $H_j=H(p_j,p_{j-1})$. We set
\begin{equation}\label{eq:process0}
\beta_j\coloneqq H_j^{(3)},
\end{equation}
and, if $\beta_j$ is nonzero, we set
\begin{equation}\label{eq:process1}
    p_{j+1}(z)\coloneqq\frac{p_{j-1}(z)-p_{j-1}(0)\,p_j(z)}{\beta_j z}
    \quad \big(\text{note that }p_{j-1}(0)=1\text{ when }j\ge 1\big)
\end{equation}
so that $H_{j+1}\coloneqq H(p_{j+1},p_{j})$. These steps can be repeated unless~$\beta_j=0$. In
Corollary~\ref{cr:FjS_i_Fj+1S} we will show that $\beta_j>0$ whenever
$F_j(z)=\frac{p_{j-1}(z)}{p_j(z)}$ represents a non-constant meromorphic $\mathcal{S}$-function.
To do this we need some auxiliary facts.

Suppose that $\beta_i\ne 0$, $i=0,1,\dots,j$ for some nonnegative $j$, such that the power
series~$p_{j-1}(z)$, $p_{j}(z)$ and $p_{j+1}(z)$ are defined according to the recurrence
formula~\eqref{eq:process1}.

\begin{lemma}\label{lemma:H_lpm}
    The identity\footnote{The notation~$\lfloor a\rfloor$ stands
        for the maximal integer not exceeding $a$.}
    \begin{equation*}
        H_j\pmat 2&3&\hdots&k&k+1\\2&3&\hdots&k&i+ 1\epmat
        = \beta_j^{\left\lfloor\frac{k}{2}\right\rfloor}
        H_{j+1}\pmat 2&3&\hdots&k-1&k\\2&3&\hdots&k-1&i\epmat,
    \end{equation*}
    holds for all $k=2,3,\dots$ and $i=k,k+1,\dots$.
\end{lemma}
\begin{proof}
    Without loss of generality we consider the case~$j=0$, since for higher values of $j$ the
    relations~\eqref{eq:process0}--\eqref{eq:process1} are analogous. In the case $k=2m$
    \begin{multline}
        \beta_0^{m} H_{1}\pmat
            2&3&\hdots&2m-1&2m\\2&3&\hdots&2m-1&i\epmat=
        \beta_0^{m} H_{1}\pmat
            1&2&\hdots&2m-1&2m\\1&2&\hdots&2m-1&i\epmat=\\
        \begin{vmatrix}
            a_0 &       a_1&       a_2 &\hdots&a_{2m-2}             & a_{i-1}\\
            0& b_1-b_0a_1&b_2-b_0a_2&\hdots&b_{2m-2}-b_0a_{2m-2}& b_{i-1}-b_0a_{i-1}\\
            0& a_0 &       a_1 &\hdots&a_{2m-3}             & a_{i-2}\\
            0& 0         &b_1-b_0a_1 &\hdots&b_{2m-3}-b_0a_{2m-3}& b_{i-2}-b_0a_{i-2}\\
            \vdots& \vdots &  \vdots  &\ddots&    \vdots &    \vdots &\\
            0& 0 & 0 &\hdots&b_{m-1}-b_0a_{m-1}& b_{i-m}-b_0a_{i-m}
        \end{vmatrix}=\\
        \begin{vmatrix}
            a_0 &a_1 &a_2&\hdots&a_{2m-2}             & a_{i-1}\\
            b_0&b_1&b_2&\hdots&b_{2m-2}& b_{i-1}\\
            0    &a_0&a_1&\hdots&a_{2m-3}             & a_{i-2}\\
            0    &b_0&b_1&\hdots&b_{2m-3}            & b_{i-2}\\
            \vdots & \vdots &  \vdots  &\ddots&    \vdots &    \vdots &\\
            0 &0 & 0 &\hdots&b_{m-1}& b_{i-m}
        \end{vmatrix}
        = H_0\pmat
            2&3&\hdots&2m&2m+1\\2&3&\hdots&2m&i+1
        \epmat,\nonumber
    \end{multline}
    here the equality $a_0=1$ has been used. For $k=2m-1$ the transformation remains the same.
\end{proof}
In particular, if we suppose that $\beta_0,\beta_1,\dots,\beta_{k-1}>0$ for $k\ge 3$, this lemma
implies
\begin{multline}\label{eq:H_jk_beta}
    H_j^{(k)}=\beta_{j}^{\left\lfloor\frac{k-1}{2}\right\rfloor}H_{j+1}^{(k-1)}=\\
    \beta_{j}^{\left\lfloor\frac{k-1}{2}\right\rfloor}\beta_{j+1}^{\left\lfloor\frac{k-2}{2}\right\rfloor}H_{j+2}^{(k-2)}
    =\dots =H_{k-3}^{(3)} \prod_{i=1}^{k-3}\beta_{i+j-1}^{\left\lfloor\frac{k-i}{2}\right\rfloor}
    =\prod_{i=1}^{k-2}\beta_{i+j-1}^{\left\lfloor\frac{k-i}{2}\right\rfloor}.
\end{multline}

The next theorem was established by Chebotarev, see~\cite{Chebotarev}
and~\cite[Ch.V~\S 1]{ChebMei}; see also the proof of M.~Schiffer and V.~Bargmann in~\cite[II.8]{Wigner}.
At the same time, it can be derived as a particular case from Nevanlinna's theory, see
\cite[Theorem~8]{AkhKrein}.
\begin{theorem}[\cite{Chebotarev,ChebMei,Wigner,AkhKrein}]\label{th:chebotarev}
    A real meromorphic function~$F(z)$ regular at the origin is an $\mathcal{R}$-function if and
    only if it has the form
    \begin{equation}\label{F_sum_form_R}
        \begin{gathered}
            F(z)=B_0 + B_1z + \sum_{1\le\nu\le\omega}\left(
                \frac {A_\nu}{z+\sigma_\nu}-\frac{A_\nu}{\sigma_\nu}
            \right),\\\text{ where }
            \sum_{1\le\nu\le\omega}\frac{|A_\nu|}{\sigma_\nu^2}<\infty,
            \ B_1>0 \text{ and } A_\nu<0,\ \sigma_\nu\in\mathbb R \text{ for } \nu=1,2\dots,\omega.
        \end{gathered}
    \end{equation}
\end{theorem}
The proof of this theorem relies on the following fact which we will use later.
\begin{lemma}[{see {\em e.g.}~\cite[Ch.VI~\S 8]{ChebMei}}]\label{lemma:cheb1}
    Let entire functions $q(z)$ and $p(z)$ have no common zeros and such that
    $F=\frac{q}{p}\in\mathcal{R}$ is not a constant. Then $F'(z)>0$ on the real line, the zeros
    of $p(z)$ and $q(z)$ are real, simple and interlacing.
\end{lemma}
The interlacing property means that between each two consequent zeros of $p(z)$ there exists a
unique root of $q(z)$ and \emph{vice versa}. The proof from \cite{ChebMei} is based on the
behaviour of meromorphic $\mathcal{R}$-functions in neighbourhoods of its zeros and poles. For
completeness, we deduce this lemma here from the partial fraction
expansion~\eqref{F_sum_form_R}.
\begin{proof}
    Let $F(z)=\frac{q(z)}{p(z)}$ have the form~\eqref{F_sum_form_R}. If $z$ is not real, then
    \[
    \frac{\Im F(z)}{\Im z} = B_1 + \sum_{1\le\nu\le\omega}\frac {-A_\nu}{|z+\sigma_\nu|^2} > 0.
    \]
    Therefore, $F(z)$ (as well as $q(z)$) has no zeros outside the real axis.

    Now from \eqref{F_sum_form_R} it follows that $F(z)$ is real and can only have simple poles.
    Since
    \[
    F'(z)=B_1 + \sum_{1\le\nu\le\omega}\frac {-A_\nu}{(z+\sigma_\nu)^2}>0,\quad z\in\mathbb{R},
    \]
    the function $F(z)$ grows between any of its two subsequent poles $z_1$ and~$z_2$ from
    $-\infty$ to $+\infty$. So there is one and only one $z_*\in(z_1,z_2)$ such that
    $F(z_*)=q(z_*)=0$. For the same reason, there exists a unique zero of~$p(z)$ between any two
    subsequent zeros of $q(z)$.
\end{proof}

The next theorem is a consequence of Grommer's theorem (see~\cite[\S 14, Satz III]{Grommer}) and
Theorem~\ref{th:chebotarev}. It can be proved by applying the Hurwitz transformation
\cite{Hurwitz} (see also~\cite[\S 6.1]{AkhKrein}, \cite[Ch.I~\S 7]{ChebMei},
\cite[Theorem~1.5]{HoltzTyaglov}, \cite{Gantmakher}) to the matrices of the Hankel forms
corresponding to~$F(z)$.
\begin{theorem}[{\emph{e.g.}~\cite[Ch.V~\S 3]{ChebMei}}]\label{th:cheb_mei}
    A meromorphic function~$F(z)=\frac{q(z)}{p(z)}$, where $p(z)$ and $q(z)$ are of the
    form~\eqref{eq:pq_exp}, is an $\mathcal R$-function if and only if there exists $l$,
    $0\le l\le\infty$, such that
    \[
    H^{(2m+1)}>0,\ m=1,2,\dots,l,\ 
    H^{(2l+3)}=H^{(2l+5)}=\dots=0.
    \]
    Moreover, $l$ is finite if and only if~$F(z)$ is a rational function with exactly $l$~poles,
    counting a pole at infinity (if exists).
\end{theorem}

Let $\beta_i\ne 0$, $i=0,1,\dots,j$ for some nonnegative $j$, and the power series~$p_{j-1}(z)$,
$p_{j}(z)$ and $p_{j+1}(z)$ be defined by the recurrence formula~\eqref{eq:process1}. Suppose
that the ratio $F_j(z)=\frac{p_{j-1}(z)}{p_j(z)}$ of formal power series converges to
\emph{a meromorphic function}. Then there exist entire functions $\tilde{p}_{j-1}(z)$ and
$\tilde{p}_j(z)$ with no common zeros such that
\[
F_j(z)=\frac{\tilde{p}_{j-1}(z)}{\tilde{p}_j(z)},\quad \tilde{p}_{j-1}(0)=p_{j-1}(0)\quad
\text{and}\quad \tilde{p}_{j}(0)=p_{j}(0)=1.
\]
Define the power series~$g(z)\coloneqq\frac{p_j(z)}{\tilde{p}_j(z)}$ satisfying $g(0)=1$. Then
\[
\quad\tilde{p}_j(z)=\frac{p_j(z)}{g(z)} \quad\text{and}\quad
\tilde{p}_{j-1}(z)=\frac{p_{j-1}(z)}{g(z)}.
\]
If $F_j\in\mathcal{R}$ is a non-constant function, then by Theorem~\ref{th:cheb_mei} the
inequality $\beta_j=\displaystyle H_j^{(3)}>0$ is satisfied. So from the
formula~\eqref{eq:process1} we find
\[
F_{j+1}(z)\coloneqq\frac{p_j(z)}{p_{j+1}(z)} =\dfrac{\beta_{j}z}{F_{j}(z)-p_{j-1}(0)}.
\]
\begin{lemma}\label{lemma:p_gcd}
    The ratio $\frac{p_{j+1}(z)}{g(z)}$ converges to the entire function
    \begin{equation}\label{eq:proc1_lt}
        \tilde{p}_{j+1}(z)\coloneqq
        \frac{\tilde{p}_{j-1}(z)-\tilde{p}_{j-1}(0)\,\tilde{p}_j(z)}{\beta_j z}.
    \end{equation}
    The pairs $(\tilde{p}_j(z),\tilde{p}_{j+1}(z))$ and
    $(\tilde{p}_{j-1}(z),\tilde{p}_{j+1}(z))$ have no common zeros.
\end{lemma}
\begin{proof}
    Dividing~\eqref{eq:process1} by $g(z)$ gives $\frac{p_{j+1}(z)}{g(z)}=\tilde{p}_{j+1}(z)$,
    that means the relation~\eqref{eq:proc1_lt} holds. Consequently,
    \begin{equation*}
        \tilde{p}_{j-1}(z)=\beta_{j} z\,\tilde{p}_{j+1}(z) + \tilde{p}_{j-1}(0)\,\tilde{p}_{j}(z).
    \end{equation*}
    Each common zero of any two summands in this equation must be a zero of the third
    summand. Since the functions $\tilde{p}_{j-1}(z)$ and $\tilde{p}_j(z)$ have no common zeros,
    the pairs $(\tilde{p}_{j-1}(z),\tilde{p}_{j+1}(z))$ and
    $(\tilde{p}_j(z),\tilde{p}_{j+1}(z))$ also have no common zeros.
\end{proof}
This lemma implies that $F_{j+1}(z)$ represents the meromorphic function
\[F_{j+1}(z)=\dfrac{\tilde{p}_j(z)}{\tilde{p}_{j+1}(z)}.\]

\begin{lemma}\label{lemma:FS_FRR}
    If the meromorphic function $F_j(z)$ is not a constant, then $F_j\in\mathcal{S}$ if and only
    if~$F_j,F_{j+1}\in\mathcal{R}$ and $F_j(0)\ge 0$.
\end{lemma}
\begin{proof}
    Let $F_j\in\mathcal{S}$, then Theorem~\ref{th:chebotarev} gives that it has the form
    \begin{gather*}
        F_j(z)=B_0 + B_1z + \sum_{1\le\nu\le\omega}\left(
            \frac {A_\nu}{z+\sigma_\nu}-\frac{A_\nu}{\sigma_\nu}
        \right) = B_0 + B_1z + z\sum_{1\le\nu\le\omega}
        \frac {(-A_\nu/\sigma_\nu)}{z+\sigma_\nu},\\\text{ where }
        \sum_{1\le\nu\le\omega}\frac{|A_\nu|}{\sigma_\nu^2}<\infty,
        \ B_0\ge 0,\ B_1>0 \text{ and } A_\nu<0,\ \sigma_\nu>0 \text{ for } \nu=1,2\dots,\omega.
    \end{gather*}
    It is enough to show that $F_{j+1}(z)$ is a well-defined $\mathcal{R}$-function. Consider the
    function
    \begin{equation*}
        G_j(z)\coloneqq\dfrac{F_j(-z)-F_j(0)}{-z}=
        B_1 + \sum_{1\le\nu\le\omega} \frac {(-A_\nu/\sigma_\nu)}{-z+\sigma_\nu}
        = B_1 + \sum_{1\le\nu\le\omega}\frac{(A_\nu/\sigma_\nu)}{z-\sigma_\nu}.
    \end{equation*}
    It has the form~\eqref{F_sum_form_R} and, hence, is a meromorphic $\mathcal{R}$-function by
    Theorem~\ref{th:chebotarev}.

    The mappings $z\mapsto \frac 1{z}$ and $z\mapsto -z$ are in the class~$\mathcal{R}^{-1}$
    ({\em i.e.} they map the upper half of the complex plane into the lower half of the complex
    plane). Since $\beta_j=H_j^{(3)}>0$ and $G_j\in\mathcal{R}$, the function composition
    \[
    \left(\frac{\beta_j}{\makebox[1em]{$\cdot$}}\circ G_j\circ(-\makebox[1em]{$\cdot$})\right)(z)
    = \frac{\beta_j}{G_j(-z)}=\dfrac{\beta_{j}z}{F_{j}(z)-p_{j-1}(0)}=F_{j+1}(z)
    \]
    is an $\mathcal{R}$-function as well.

    Conversely, let $F_j,F_{j+1}\in\mathcal{R}$ and $F_j(0)\ge 0$. The inequality
    $\beta_j>0$ holds, therefore $F_{j+1}(z)\not\equiv 0$ and the meromorphic function
    \[
    G_j(z)\coloneqq\frac{\beta_j}{F_{j+1}(-z)}=\dfrac{F_j(-z)-F_j(0)}{-z}
    \]
    is an $\mathcal{R}$-function. On one hand, Theorem~\ref{th:chebotarev} gives
    \begin{gather*}
        F_j(z)= B_0 + B_1z + z\sum_{1\le\nu\le\omega}
        \frac {(-A_\nu/\sigma_\nu)}{z+\sigma_\nu},\text{ where}\\
        \sum_{1\le\nu\le\omega}\frac{|A_\nu|}{\sigma_\nu^2}<\infty,
        \ B_1>0 \text{ and } A_\nu<0,\ \sigma_\nu\in\mathbb{R}
        \text{ for } \nu=1,2\dots,\omega,\\
    \end{gather*}
    such that
    \[
    G_j(z)=\dfrac{F_j(-z)-F_j(0)}{-z}
    = B_1 + \sum_{1\le\nu\le\omega} \frac {(A_\nu/\sigma_\nu)}{z-\sigma_\nu}.
    \]
    On the other hand, Theorem~\ref{th:chebotarev} states that each $\mathcal{R}$-function has
    negative residues at its poles. That is, $\frac{A_\nu}{\sigma_\nu}<0$ for all $\nu$ since
    $G_j\in\mathcal{R}$. Therefore, the poles $-\sigma_\nu$, $\nu=1,2\dots,\omega$, of the
    function $F_j$ are negative.  Consequently,~$F_j\in\mathcal{S}$.
\end{proof}

\begin{corollary}\label{cr:FjS_i_Fj+1S}
    Suppose that for some $j\ge 0$ the function $F_j(z)$ is in the class~$\mathcal{S}$.
    Then~$\beta_j\ge 0$. The inequality~$\beta_j>0$ implies $F_{j+1}\in\mathcal{S}$, while the
    equality~$\beta_j=0$ implies that $F_j(z)$ is constant.
\end{corollary}
\begin{proof}
    If $F_j(z)$ is a constant then $\beta_j=H_j^{(3)}=0$ by Theorem~\ref{th:cheb_mei}. Let
    $F_j(z)$ be a non-constant $\mathcal{S}$-function. Applying Lemma~\ref{lemma:FS_FRR} to it
    gives $F_{j+1}\in\mathcal{R}$. Consequently, $\beta_{j+1}=0$ if $F_{j+1}(z)$ is a constant
    and $\beta_{j+1}>0$ if it is not. Moreover, we have $F_{j+1}(0)=1>0$. Thus, the corollary
    holds in the cases of constant $F_j(z)$ or~$F_{j+1}(z)$.

    Suppose that $\beta_j,\beta_{j+1}>0$. Then Lemma~\ref{lemma:H_lpm} implies that
    \begin{equation}\label{eq:Hj+2_Hj}
        H_j^{(2m+3)}=\beta_j^{\left\lfloor\frac{2m+2}{2}\right\rfloor}
        \beta_{j+1}^{\left\lfloor\frac{2m+1}{2}\right\rfloor}H_{j+2}^{(2m+1)}
        =\beta_j^{m+1}\beta_{j+1}^m H_{j+2}^{(2m+1)},\quad m=1,2,\dots.
    \end{equation}
    That is, for each natural~$m$ the sign of $H_{j+1}^{(2m+1)}$ coincides with the sign of
    $H_j^{(2m+3)}$. Since $F_j\in\mathcal{R}$, Theorem~\ref{th:cheb_mei}
    yields~$F_{j+2}\in\mathcal{R}$. That is, $F_{j+1}\in\mathcal{S}$ by Lemma~\ref{lemma:FS_FRR}.
\end{proof}

\begin{theorem}\label{th:cheb_mei_pos}
    A meromorphic function~$F(z)=\frac{q(z)}{p(z)}$, where $p(z)$ and $q(z)$ are series of the
    form~\eqref{eq:pq_exp}, is an $\mathcal{S}$-function if and only if there exists $\omega$,
    $2\le \omega\le\infty$, such that
    \begin{equation}\label{eq:minorseq}
        H^{(k)}>0,\ k=2,3,\dots,\omega\text{\quad and\quad}
        H^{(\omega+1)}=H^{(\omega+2)}=\dots=0.
    \end{equation}
    Moreover, $\omega$ is finite if and only if $F(z)$ is a rational function with exactly
    $\left\lfloor\frac{\omega-1}{2}\right\rfloor$ poles, counting a pole at infinity (if
    exists).
\end{theorem}
\begin{proof}
    By definition,~$H_0^{(2)}=H^{(2)}=1>0$. Denote $p_0(z)\coloneqq p(z)$ and
    $\relpenalty=10000 p_{-1}(z)\coloneqq q(z)$ such that $F(z)=F_0(z)$. From the recurrence
    formul\ae~\eqref{eq:process0}--\eqref{eq:process1} we obtain the sequences
    $(p_j)_{j=-1}^{\omega-2}$ and $(\beta_j)_{j=0}^{\omega-2}$, where $\beta_j\ne 0$ for all
    $j=0,1,\dots \omega-3$ and $\relpenalty=10000 2\le \omega\le\infty$. Whenever $\omega<\infty$ we also have
    $\beta_{\omega-2}=0$.

    Suppose that $F_0\in\mathcal{S}$. Then $\beta_j>0$ for all $j=0,1,\dots \omega-3$ by
    Corollary~\ref{cr:FjS_i_Fj+1S}. Furthermore, the identity~\eqref{eq:H_jk_beta} gives
    \begin{equation}\label{eq:H_j_beta}
        H^{(j)}=H_0^{(j)}=\prod_{i=1}^{j-2}\beta_{i-1}^{\left\lfloor\frac{j-i}{2}\right\rfloor}>0,
        \quad j=3,4,\dots \omega.
    \end{equation}
    Let $\omega<\infty$, then $F_{\omega-2}(z)$ is a constant by
    Corollary~\ref{cr:FjS_i_Fj+1S}. Therefore, we have
    $H_{\omega-2}^{(3)}=H_{\omega-2}^{(4)}=\dots=0$ since all these minors contain proportional
    rows. By the identity~\eqref{eq:H_jk_beta}, this is equivalent to
    $H^{(\omega+1)}=H^{(\omega+2)}=\dots=0$.
    
    So we obtained that $F\in\mathcal{S}$ implies~\eqref{eq:minorseq}. The number of poles the
    function~$F(z)$ has can be determined from Theorem~\ref{th:cheb_mei}.
    
    Now suppose that the conditions~\eqref{eq:minorseq} hold. If~$\omega=2$ then
    $H^{(3)}=H^{(4)}=\dots=0$ and the assertion of this theorem is equivalent to
    Theorem~\ref{th:cheb_mei}. In the case of $3\le \omega\le\infty$ we have $\beta_0>0$, so by
    Lemma~\ref{lemma:H_lpm},
    \begin{align*}
        &H_1^{(2m+1)}=\beta_0^{-\left\lfloor\frac{2m+1}{2}\right\rfloor} H^{(2m+2)}>0,&
        &m=1,2,\dots,\left\lfloor\frac{\omega}2\right\rfloor -1,
        &\text{and}\quad\\
        &H_1^{(2m+1)}=\beta_0^{-\left\lfloor\frac{2m+1}{2}\right\rfloor} H^{(2m+2)}
        =0,& &m\ge\left\lfloor\frac{\omega}2\right\rfloor.&
    \end{align*}
    Hence, the functions $F(z)$ and $F_1(z)$ are $\mathcal{R}$-functions by
    Theorem~\ref{th:cheb_mei}, and Lemma~\ref{lemma:FS_FRR} yields~$F\in\mathcal{S}$.
\end{proof}

\subsection{$\mathcal S$-functions as continued fractions}\label{sec:S-functions-as-CF}

A continued fraction of the form
\begin{equation}\label{F0asCCF}
    \begin{gathered}
        F(z) = c_0+\frac{c_1 z^{r_1}}1 \,\underset{+}{}\,
        \frac{c_2 z^{r_2}}1 \,\underset{+}{}\,
        \frac{c_3 z^{r_3}}1 \,\underset{+\cdots+}{}\,
        \frac{c_\omega z^{r_\omega}}1,\quad\text{where}\\
        c_j\ne 0\quad\text{and}\quad r_j\in\mathbb{N}
        \quad\text{for}\quad j=1,\dots,\omega,\quad 0\le \omega\le \infty,
    \end{gathered}
\end{equation}
is called a \emph{(general) $C$-fraction}. The special case of~\eqref{F0asCCF} that corresponds
to $\relpenalty=10000 r_j=1$ for all $j=1,\dots,\omega$ is called a \emph{regular $C$-fraction}.
Continued fractions of the form~\eqref{F0asCCF} are able to represent power series uniquely,
that is the following fact is true.

\begin{theorem}[{\cite{LeightonScott}, see also~\cite[\S 21,
        S\"atze~3.2--3.5,~3.24]{Perron}}] \label{th:c_fractions}
    Each (formal) power series $F(z)=\sum_{k=0}^\infty s_kz^k$ corresponds to a fraction of the
    form~\eqref{F0asCCF}. This correspondence is set by the following sequence of relations
    \begin{equation}\label{F0asCCF_corr}
        F_0(z) = F(z), \quad c_0=F(0),
        \quad F_{i}(z)=\frac{c_i z^{r_i}}{F_{i-1}(z)-F_{i-1}(0)},\quad i=1,2,\dots,\omega,
    \end{equation}
    where $\omega\le\infty$ is such that $F_{i-1}(z)\not\equiv F_{i-1}(0)$ for $i-1<\omega$ and
    $F_\omega(z)\equiv F_\omega(0)$. The exponents $r_i$ are positive integers chosen together
    with the complex constants $c_i$ in such a way that $F_{i}(0)=1$.

    If two $C$-fractions (finite or infinite) of the form~\eqref{F0asCCF} correspond to the same
    power series, then they coincide. A $C$-fraction is finite if and only if it corresponds to
    a rational function (and, hence, represents that function).

    Moreover, if an infinite continued fraction of the form~\eqref{F0asCCF} converges uniformly
    in a closed region~$T$ containing the origin in its interior, it represents a regular analytic
    non-rational function of~$z$ throughout the interior of~$T$. Further, the corresponding power
    series converges to the same function in and on the boundary of the largest circle which can
    be drawn with its center at the origin, lying wholly within~$T$.
\end{theorem}

Suppose that $F(z)=\frac{q(z)}{p(z)}$ is an $\mathcal S$-function. We again denote
$F_0(z)\coloneqq F(z)$, $p_0(z)\coloneqq p(z)$ and $p_{-1}(z)\coloneqq q(z)$ and use the
recurrence formul\ae~\eqref{eq:process0}--\eqref{eq:process1} to obtain the sequences
$(p_j)_{j=-1}^{\omega}$ and $(\beta_j)_{j=0}^{\omega}$, where $\beta_j\ne 0$ for all
$j=0,1,\dots \omega-1$ and $-1\le \omega\le\infty$. In the case $\omega<\infty$ we also have
$\beta_{\omega}=0$.

For each $j=0,1,\dots \omega-1$, we apply Corollary~\ref{cr:FjS_i_Fj+1S}, obtaining $\beta_j>0$
and $F_j\in\mathcal{S}$. If $\omega$ is a finite number, then $F_{\omega}$ is a constant. From
the relation~\eqref{eq:process1} we have
\begin{equation}\label{eq:FjasFCF}
    F_j(z)=\frac {p_{j-1}(z)}{p_j(z)}
    = p_{j-1}(0)+\frac{\beta_jz}{F_{j+1}(z)},\quad j=0,1,\dots,\omega-1.
\end{equation}
These formul\ae\ can be combined into the continued fractions
\begin{align}\label{eq:F0asICF}
    F(z)=F_0(z) & = b_0+\frac{\beta_0 z}1 \,\underset{+}{}\,
    \frac{\beta_1 z}1 \,\underset{+}{}\,
    \frac{\beta_2 z}1 \,\underset{+\cdots+}{}\,
    \frac{\beta_{\omega-1} z}1\quad\text{and}
    \\ \label{eq:FjasICF}
    F_j(z) &= 1+\frac{\beta_j z}1 \,\underset{+}{}\,
    \frac{\beta_{j+1} z}1 \,\underset{+}{}\,
    \frac{\beta_{j+2} z}1 \,\underset{+\cdots+}{}\,
    \frac{\beta_{\omega-1} z}1,
\end{align}
where $j=0,1,\dots,\omega-1$ and $\beta_i>0$ for $i=0,1,\dots,\omega-1$. These are regular
$C$-fractions, and the relations~\eqref{eq:FjasFCF} set the correspondence
satisfying~\eqref{F0asCCF_corr}. That is, the continued fractions in \eqref{eq:F0asICF}
and~\eqref{eq:FjasICF} corresponds to $F(z)$ and $F_j(z)$ for all $j$, respectively, by
Theorem~\ref{th:c_fractions}. In particular, they are finite if and only if $F(z)$ is rational.

Furthermore, there is a power series $g(z)$ such that $\tilde{p}_{j}(z):=\frac{p_j(z)}{g(z)}$
are entire functions for $j=-1,0,\dots,\omega$, and for $j\ge 0$ \ $\tilde{p}_{j-1}(z)$ and
$\tilde{p}_{j}(z)$ have no common zeros (see Lemma~\ref{lemma:p_gcd}). Observe that the
relations~\eqref{eq:FjasFCF}, \eqref{eq:F0asICF} and \eqref{eq:FjasICF} remain the same, if we
replace all the series $p_{j}(z)$ by the functions $\tilde{p}_j(z)$.

It is convenient to study the continued fractions~\eqref{eq:F0asICF} and~\eqref{eq:FjasICF},
using the following.
\begin{theorem}[{Stieltjes,~{\cite[n\textsuperscript{os}~68--69]{Stieltjes}; see
            also~\cite{Perron,Wall}}\footnote{The separate convergence of numerators and
            denominators was shown by \'Sleszy\'nski in~\cite{Sleszinski}.}}]
    \label{th:Stieltjes} Let $b_0\ge 0$. A sequence of positive
    numbers~$\beta_0,\beta_1,\dots,\beta_\omega$, $-1\le\omega\le\infty$, has a finite sum if
    and only if the continued fraction~\eqref{eq:F0asICF} converges to a meromorphic
    $\mathcal{S}$-function and its partial numerators and denominators converge to
    coprime\footnote{This fact was obtained by Maillet in~\cite{Maillet}; see
        also~\cite[p.150]{Perron}.} entire functions of genus~$0$. That is, if the $j$th
    convergent (approximant) to~$F(z)$ is denoted by~$\frac{Q_{j}(z)}{P_{j}(z)}$, then for
    $j\to\infty$ we have
    \begin{gather*}
        P_j(z)\to p(z),\quad Q_j(z)\to q(z)
        \quad\text{and}\\
        b_0+\frac{\beta_0 z}1 \,\underset{+}{}\,
        \frac{\beta_1 z}1 \,\underset{+\cdots+}{}\, \frac{\beta_{j-1} z}1 = 
        \frac{Q_j(z)}{P_j(z)}\to \frac{q(z)}{p(z)}=F(z),
    \end{gather*}
    where $p(z)$ and $q(z)$ are coprime entire functions of genus~$0$. The convergence is
    uniform on compact subsets of~$\mathbb{C}$ containing no poles of the function $F(z)$.
\end{theorem}
To apply this theorem we need to distinguish the case of~$\sum_{j=0}^\infty\beta_j<\infty$
for~$\relpenalty=10000\omega=\infty$.

\begin{lemma}\label{lemma:decomp_decrease}
    Let the functions $\tilde{p}_1(z)$ and $\tilde{p}_{0}(z)$ be entire of genus~$0$, coprime
    and such that their ratio $\mathcal{S}\ni F_1\coloneqq\frac{\tilde{p}_{0}}{\tilde{p}_1}$ is
    not rational. Then $\tilde{p}_j(z)\to 1$ as $j\to\infty$ uniformly on compact subsets
    of~$\mathbb{C}$, and $\sum_{j=1}^\infty\beta_j<\infty$.
\end{lemma}
\begin{proof}
    According to~\eqref{eq:proc1_lt}, all the functions~$\tilde{p}_j(z)$, $j=0,1,\dots$, are entire
    of genus~$0$ and $\tilde{p}_j(0)=1$. Moreover, $\tilde{p}_0,\dots,\tilde{p}_j$ are coprime (by
    Lemma~\ref{lemma:p_gcd}) and hence have only negative zeros (since by
    Corollary~\ref{cr:FjS_i_Fj+1S} $F_{j}=\frac{\tilde{p}_{j-1}}{\tilde{p}_j}\in\mathcal{S}$).
    Therefore, the following representation is valid for $j=0,1,\dots$
    \begin{equation*}
        \tilde{p}_{j}(z)=\sum_{k=0}^\infty a_k^{(j)} z^k
        =\prod_{\nu=1}^\infty \left(1+\frac{z}{\sigma_\nu^{(j)}}\right),
    \end{equation*}
    where $0<\sigma_1^{(j)}\le\sigma_2^{(j)}\le\dots$ and
    $\sum_{\nu=1}^\infty \frac{1}{\sigma_\nu^{(j)}}<\infty$. The coefficients $a_k^{(j)}$,
    $k=1,2,\dots$, are equal to
    \begin{equation}
        a_k^{(j)} =
        \sum_{i_1=1}^\infty\sum_{\substack{i_2=1\\i_2\notin\left\{i_1\right\}}}^\infty\dots
        \sum_{\substack{i_k=1\\i_k\notin\left\{i_1,i_2,\dots,i_{k-1}\right\}}}^\infty\frac
        {1}{\sigma_{i_1}^{(j)}\sigma_{i_2}^{(j)}\cdots\sigma_{i_k}^{(j)}}.\label{eq:ak_as_sum}
    \end{equation}
    Note that these sums are convergent, since
    \footnote{In fact, even an estimate stronger than~\eqref{eq:ak_bk_are_bounded} is valid
        (\emph{cf.}~\cite[p.~105]{Sleszinski}). For each tuple of distinct numbers
        $(i_1,i_2,\dots,i_k)$ there is only one summand
        $\left(\sigma_{i_1}^{(j)}\sigma_{i_2}^{(j)}\cdots\sigma_{i_k}^{(j)}\right)^{-1}$ in the
        right-hand side of~\eqref{eq:ak_as_sum}. At the same time, the sum
        \[
        \sum_{i_1=1}^\infty\sum_{i_2=1}^\infty\dots
        \sum_{i_k=1}^\infty\frac
        1{\sigma_{i_1}^{(j)}} \frac 1{\sigma_{i_2}^{(j)}}\cdots\frac 1{\sigma_{i_k}^{(j)}}
        \]
        contains exactly~$k!$ summands of this form. Therefore,
        \[a_k^{(j)}< \frac 1{k!} \left(a_1^{(j)}\right)^{k} \text{ for }k=2,3,4,\dots.\]
    }
    \begin{equation}\label{eq:ak_bk_are_bounded}
        a_k^{(j)}<
        \sum_{i_1=1}^\infty\sum_{i_2=1}^\infty\dots
        \sum_{i_k=1}^\infty\frac
        1{\sigma_{i_1}^{(j)}} \frac 1{\sigma_{i_2}^{(j)}}\cdots\frac 1{\sigma_{i_k}^{(j)}}
        = \left(a_1^{(j)}\right)^{k}
        \text{ when }k=2,3,4,\dots.
    \end{equation}
    
    By Lemma~\ref{lemma:cheb1} the zeros of $\tilde{p}_j(z)$ and $\tilde{p}_{j-1}(z)$ (which are
    negative) must be simple and interlacing. In addition, $F'_{j}(z)>0$ for real $z$ implying
    that $0<\sigma_1^{(j-1)}<\sigma_1^{(j)}$. Hence
    \begin{equation}
        0<\sigma_1^{(j-1)}<\sigma_1^{(j)}<\sigma_2^{(j-1)}<\sigma_2^{(j)}
        <\sigma_3^{(j-1)}<\sigma_3^{(j)}<\dots.\label{eq:pq_zeros}
    \end{equation}
    Now we estimate the expression~\eqref{eq:ak_as_sum} using the
    inequalities~\eqref{eq:pq_zeros} and obtain that $0\le a_k^{(j)}<a_k^{(j-1)}$ for
    $j=1,2,\dots$, {\em i.e.} the sequence of positive numbers $a_k^{(j)}$ decreases in $j$ for
    fixed $k$. Therefore, there exists a finite~$\lim_{j\to\infty}a_k^{(j)}\ge 0$ dependent on
    $k$.

    At the same time, the equality~\eqref{eq:proc1_lt} implies that the first Taylor coefficient
    $a_1^{(j)}$ for any $j$ has the form
    \[
    a_1^{(j-1)}=\beta_{j}+a_1^{(j)}.
    \]
    Consequently,
    \begin{equation}\label{eq:beta_finite}
        a_1^{(0)}=\sum_{i=1}^{j}\beta_i+a_1^{(j)}\quad\text{and}\quad\sum_{j=1}^{\infty}\beta_j
        =a_1^{(0)}-\lim_{j\to\infty}a_1^{(j)}.
    \end{equation}
    Therefore, the series~$\sum_{j=0}^\infty\beta_j$ converges.

    As a consequence, for an arbitrary positive number $R$ there exists an integer $j_0(R)$ such
    that $\beta_jR<\frac 14$ for all $j\ge j_0$. By virtue of Worpitzky's test (as it stated
    in~\cite[p.~45]{Wall}, see also~\cite{Worpitzky}) the continued fraction
    \[
    F_{j_0}(z) = 1+\frac{\beta_{j_0} z}1 \,\underset{+}{}\,
    \frac{\beta_{j_0+1} z}1 \,\underset{+}{}\,
    \frac{\beta_{j_0+2} z}1 \,\underset{+\cdots}{}
    \]
    converges to an analytic function uniformly in the disk $|z|<R$. This analytic function
    coincides with~$F_{j_0}(z)$ (since $F_{j_0}(z)$ corresponds to the continued fraction,
    see~\eqref{eq:FjasICF}). Therefore, $\tilde{p}_{j_0}$ has no zeros in this disk, that is
    $R<\sigma_1^{(j_0)}<\sigma_1^{(j)}$, $j\ge j_0$. Letting $R$ tend to infinity, we obtain
    $\lim_{j\to\infty}\sigma_1^{(j)}=\infty$. According to~\eqref{eq:ak_as_sum} we have
    \[
        a_1^{(j)}=\sum_{i=1}^\infty\frac{1}{\sigma_{i}^{(j)}}
    \]
    and each term in this series monotonically tends to zero as $j\to\infty$. For any
    $\varepsilon>0$ there exists $N$ such that
    \[    
    \sum_{i=N+1}^\infty\frac{1}{\sigma_{i}^{(j_0)}}<\varepsilon,\text{ which for $j\ge j_0$ gives}
    \sum_{i=N+1}^\infty\frac{1}{\sigma_{i}^{(j)}}\le
    \sum_{i=N+1}^\infty\frac{1}{\sigma_{i}^{(j_0)}}<\varepsilon.
    \]
    On the other hand, $\displaystyle\lim_{j\to\infty}\sum_{i=1}^N\frac{1}{\sigma_{i}^{(j)}}=0$,
    so the coefficient $a_1^{(j)}$ also vanishes. Now from~\eqref{eq:ak_bk_are_bounded} we obtain
    that $\tilde{p}_j(z)\to 1$ as $j\to\infty$ uniformly on compact subsets of~$\mathbb C$.
\end{proof}

\begin{corollary}\label{cr:norm_H}
    Under the assumptions of Lemma~\ref{lemma:decomp_decrease} there exists a positive number
    $M$ independent of~$j$ such that
    \[
    \|H(\tilde{p}_{j+1},\tilde{p}_j)\|_1 < M\quad\text{for}\quad j=0,1,\dots,
    \]
    where the matrix $H(\tilde{p}_{j+1},\tilde{p}_j)$ is defined by~\eqref{eq:Hpq_def}, and
    \[  
    \| H(\tilde{p}_{j+1},\tilde{p}_j)-H(1,1)\|_1\xrightarrow{j\to\infty} 0.
    \]
\end{corollary}
\begin{proof} Since $a_k^{(j)}\ge a_k^{(j+1)}\ge 0$ for all $j,k=0,1,\dots$, we have
    \[
        \|H(\tilde{p}_{j+1},\tilde{p}_j)\|_1 = \max\left\{\sum_{k=0}^\infty
            a_k^{(j+1)} ,\sum_{k=0}^\infty a_k^{(j)}\right\} =
        \sum_{k=0}^\infty a_k^{(j)}\le\sum_{k=0}^\infty a_k^{(0)}=\tilde{p}_0(1)<\infty.
    \]
    Observe that $a_0^{(j)}=1$ whenever $j\ge 0$. Therefore, by
    Lemma~\ref{lemma:decomp_decrease} we obtain the required
    \begin{multline*}
        \|H(\tilde{p}_{j+1},\tilde{p}_j) - H(1,1)\|_1 = \max\left\{\sum_{k=1}^\infty a_k^{(j+1)}
            ,\sum_{k=1}^\infty a_k^{(j)}\right\} =\\\sum_{k=1}^\infty
        a_k^{(j)}=\tilde{p}_j(1)-1\xrightarrow{j\to\infty}0.
    \end{multline*}
\end{proof}
{
    \renewcommand{\thetheorem}{\ref{cr:main2}}
    \begin{corollary}
        Let~$F(z)=\frac{q(z)}{p(z)}$ be a meromorphic $\mathcal{S}$-function, where the entire
        functions $p(z)$ and $q(z)$ are of genus~$0$. Then it can be expanded into a uniformly
        convergent continued fraction of the form~\eqref{eq:F0asICF} with the coefficients $b_0$ and
        $(\beta_j)_{j=0}^{\omega-1}$, $\sum_{j=0}^{\omega-1}\beta_j<\infty$, given
        by~\eqref{eq:process0}. No other continued fractions of the form~\eqref{F0asCCF} can
        correspond to the Taylor series of~$F(z)$.
    \end{corollary}
    \addtocounter{theorem}{-1}
}
\begin{proof}
    If $F=\frac{p}{q}\in\mathcal{S}$ then $F(z)$ can be formally developed into the continued
    fraction~\eqref{eq:F0asICF} with the coefficients $b_0=F(0)$ and
    $(\beta_j)_{j=0}^{\omega-1}$ given by~\eqref{eq:process0}. By
    Lemma~\ref{lemma:decomp_decrease}, the coefficients of this continued fraction satisfy the
    condition~$\sum_{j=0}^{\omega-1}\beta_j<\infty$. Consequently, Theorem~\ref{th:Stieltjes}
    implies that \eqref{eq:F0asICF} converges uniformly on compact sets containing no poles of
    its limiting function. So the rest of the proof comes from Theorem~\ref{th:c_fractions}: the
    continued fraction~\eqref{eq:F0asICF} corresponds to and converges to $F(z)$, and there is
    no other continued fraction of the form~(\ref{F0asCCF}) corresponding to $F(z)$.
\end{proof}

\section{Total nonnegativity of Hurwitz-type matrices}
This section contains the proof of Theorem~\ref{th:main1} and Corollary~\ref{cr:main2}, preceded by
several auxiliary facts.

Suppose that meromorphic functions $p$ and $q$ are regular at the origin and have the Taylor
expansion~\eqref{eq:pq_exp}. Consider the Hurwitz-type matrix~$H(p,q)$, defined
by~\eqref{eq:Hpq_def}.

\begin{lemma}\label{lemma:dec_general}
    If $\frac {q}{p}\in\mathcal S$ and there exists a meromorphic function $g(z)$, such that the
    ratios $\tilde{p}(z)\coloneqq\frac{p(z)}{g(z)}$ and $\tilde{q}(z)\coloneqq\frac{q(z)}{g(z)}$
    are entire of genus~$0$ and coprime, then the matrix $H(p,q)$ can be factored as
    in~\eqref{eq:H_inf_fact}, where the numbers $\beta_j$, $j=0,1,\dots$, are given
    by~\eqref{eq:process0} (possibly followed by zeros). Moreover,
    \[
    \big\|H(p,q) - J(b_0,\beta_0)\,J(1,\beta_1) \cdots
                J(1,\beta_j)\,H(1,1)\,T(g)\big\|_\rho\xrightarrow{j\to\infty}0,
    \]
    where $\rho$, $0<\rho\le 1$, is such that $g(z)$ has no poles in the disk $|z|\le\rho$.
\end{lemma}
\begin{proof}
    For any two matrices $A=(a_{kl})_{k,l=1}^\infty$ and $B=(b_{kl})_{k,l=1}^\infty$ such that
    $\|A\|_1<\infty$ and $\|B\|_\rho<\infty$ the following estimate (implying the existence of the
    product~$AB$) is true
    \begin{multline}\label{eq:norm_1rho}
        \infty>\|A\|_1\|B\|_\rho
        =\sup_{1\le k<\infty}\sum_{l=1}^\infty |a_{kl}|\sup_{1\le m<\infty}\sum_{j=1}^\infty |b_{mj}|\rho^{j-1}
        \ge\\
        \sup_{1\le k<\infty}\sum_{l=1}^\infty |a_{kl}|\sum_{j=1}^\infty |b_{lj}|\rho^{j-1}
        \ge \sup_{1\le k<\infty}\sum_{j=1}^\infty\left|\sum_{l=1}^\infty a_{kl}b_{lj}\right|\rho^{j-1}
        = \|AB\|_\rho.
    \end{multline}

    Now we note that the decomposition
    \begin{equation}\label{eq:Hpjqj_dec}
        H(p,q)=H(\tilde{p},\tilde{q})\,T(g)
    \end{equation}
    is valid. It can be checked by the straightforward multiplication.

    Denote~$p_0(z)\coloneqq\tilde{p}(z)$ and $p_{-1}(z)\coloneqq\tilde{q}(z)$. We are now using
    the algorithm~\eqref{eq:process0}--\eqref{eq:process1} to construct the (longest possible)
    sequence~$(p_j)_{j=-1}^{\omega}$ of entire functions, $\relpenalty=10000 0\le \omega\le\infty$. By
    Corollary~\ref{cr:FjS_i_Fj+1S}, the corresponding numbers~$(\beta_j)_{j=0}^{\omega}$ satisfy
    $\beta_j>0$ for all $j=0,1,\dots \omega-1$. In the case of finite $\omega$ we have
    $p_{\omega-1}(z)\equiv p_{\omega-1}(0)$, $p_{\omega}(z)\equiv 1$ and $\beta_{\omega}=0$; we
    extend the latter equality by $\beta_{\omega+1}=\beta_{\omega+2}=\dots=0$.

    The identity \eqref{eq:Hpjqj_dec} implies the factorization~\eqref{eq:H_fin_fact} in the
    case of~$q(z)\equiv q(0)p(z)$ (corresponding to~$\omega=0$). Suppose that $\omega>0$. If we
    expand $p_i(z)$ and $p_{i-1}(z)$, $i=0,1,\dots j< \omega$, as follows
    \[
    p_i(z)=\sum_{k=0}^\infty\,c_kz^k
    \quad\text{and}\quad
    p_{i-1}(z)=\sum_{k=0}^\infty\,d_kz^k,
    \]
    then the matrix~$H_{i+1}\coloneqq H(p_{i+1},p_i)$ takes the form
    \begin{equation*}
        H_{i+1}=
        {\begin{pmatrix}
                {c_0} & {c_1} & {c_2} & {c_3} & \hdots\\
                0 & \frac 1{\beta_i}(d_1-d_0c_1) &  \frac 1{\beta_i}(d_2-d_0c_2) &
                \frac 1{\beta_i}(c_0d_3-d_0c_3) & \hdots\\
                0 &{c_0} & {c_1} & {c_2} & \hdots\\
                0 & 0 & \frac 1{\beta_i}(d_1-d_0c_1) &  \frac 1{\beta_i}(d_2-d_0c_2) & \hdots\\
                0 & 0 &{c_0} & {c_1} & \hdots\\
                \vdots & \vdots & \vdots & \vdots & \ddots
            \end{pmatrix}}.
    \end{equation*}
    Left-multiplying this matrix by
    \begin{equation*}
        J(d_0,\beta_i)={\begin{pmatrix}
                d_0 & \beta_i & 0 & 0 & 0 & 0 & \hdots\\
                0 & 0 & 1 & 0 & 0 & 0 & \hdots\\
                0 & 0 & d_0 & \beta_i & 0 & 0 & \hdots\\
                0 & 0 & 0 & 0 & 1 & 0 & \hdots\\
                0 & 0 & 0 & 0 & d_0 & \beta_i & \hdots\\
                \vdots & \vdots & \vdots & \vdots & \vdots & \vdots & \ddots
            \end{pmatrix}},
    \end{equation*}
    we obtain~$H_i$. On putting $i$ successively equal to $0$, $1$, \dots, $j$ and
    applying~\eqref{eq:Hpjqj_dec} we find that
    \begin{equation}\label{eq:H_factor}
    H(p,q)=J(b_0,\beta_0)\,J(1,\beta_1)\cdots J(1,\beta_j)\,H(p_{j+1},p_j)\,T(g).
    \end{equation}
    The finite product of the matrices is well defined and associative, since the matrices
    $J(1,\cdot)$ and $H(1,1)$ have at most two nonzero entries in each row and column. So if
    $\omega$ is a finite number, then for~$j=\omega-1$ the equality~\eqref{eq:H_factor}
    coincides with~\eqref{eq:H_fin_fact}. Since $\|T(g)\|_\rho=g(\rho)<\infty$, from
    \eqref{eq:J(c,0)H(1,1)},~\eqref{eq:norm_1rho} and Corollary~\ref{cr:norm_H} we obtain the
    assertion of the theorem for~$\omega<\infty$.

    Suppose that $\omega=\infty$ and let us prove that the difference between the product
    in~\eqref{eq:H_factor} and the right-hand side of~\eqref{eq:H_inf_fact} converges to zero as
    $j\to\infty$. There exists an index $j_0\ge 1$ such that
    \begin{equation}\label{eq:beta_less_1_2}
        \sum_{j=j_0}^\infty\beta_j<\frac{1}{2}.
    \end{equation}
    Let
    \[
    V\coloneqq J(b_0,\beta_0)\,J(1,\beta_1) \cdots J(1,\beta_{j_0-1}) \text{ \ and \ }
    U_j\coloneqq J(1,\beta_{j_0})\dotsb J(1,\beta_j)
    \]
    for $j=j_0,j_0+1,\dots$. Then we can express the equality~\eqref{eq:H_factor} as follows
    \[H(p,q)= V\,U_j\,H(p_{j+1},p_{j})\,T(g).\]

    The matrix~$U_j=J(1,\beta_{j_0})\dotsb J(1,\beta_j)$ is upper triangular and has no negative
    entries since it is a product of upper triangular matrices with nonnegative entries. The
    diagonal elements of~$U_j$ are the products of corresponding diagonal elements of
    $J(1,\beta_{j_0})$, \dots, $J(1,\beta_j)$ and, thus, are equal to~$1$ on the odd rows
    and~$0$ on the even ones.
    
    More specifically, denote the entries of~$U_j$ by~$u^{(j)}_{kl}$ so that
    \[ U_j=\left(u^{(j)}_{kl}\right)_{k,l=1}^\infty.\]
    For $j\ge j_0$, $k,m=1,2,\dots$ the equality $U_{j+1}=U_j\,J(1,\beta_{j+1})$ implies
    \[
    u^{(j+1)}_{k,1} = u^{(j)}_{k,1},\qquad
    u^{(j+1)}_{k,2m} = u^{(j)}_{k,2m-1}\,\beta_{j+1},\qquad
    u^{(j+1)}_{k,2m+1} = u^{(j)}_{k,2m} + u^{(j)}_{k,2m+1}.
    \]
    The following entries for all $j\ge j_0$ must be zero
    \[
    u^{(j)}_{k,2m-1} = u^{(j)}_{k,2m} = 0,\quad m=1,2,\dots,\quad k=2m,2m+1,\dots.
    \]

    For $j=j_0$ we have $U_j=J(1,\beta_j)$, so the nonzero entries of $U_j$ are only
    \[
    u^{(j)}_{2m-1,2m}=\beta_j,\quad u^{(j)}_{2m,2m+1}=1\quad
    \text{and}\quad u^{(j)}_{2m+1,2m+1}=1, \quad\text{where}\quad m=1,2,\dots.
    \]
    Consequently, the following estimate is valid
    \begin{equation}\label{eq:U_bound}
        u^{(j)}_{k,2m}+u^{(j)}_{k,2m+1}
        \le\left({\sum_{i=j_0}^{j}}\beta_i\right)^{\mathrlap{m-\left\lfloor\frac{k}{2}\right\rfloor}},
        \quad m=1,2,\dots,\quad k=1,2,\dots,2m.
    \end{equation}
    Suppose that~\eqref{eq:U_bound} holds for some~$j\ge j_0$, then for $k\le 2m$ we have
    \begin{multline}\nonumber
        u^{(j+1)}_{k,2m}+u^{(j+1)}_{k,2m+1} = u^{(j)}_{k,2m-1}\,\beta_{j+1} + \left(u^{(j)}_{k,2m} +
            u^{(j)}_{k,2m+1}\right)\le\\
        (u^{(j)}_{k,2m-2}+u^{(j)}_{k,2m-1})\,\beta_{j+1} +
        \left(u^{(j)}_{k,2m} + u^{(j)}_{k,2m+1}\right)
        \le\\
        \left( \beta_{j+1} + \sum_{i=j_0}^{j}\beta_i\right)
        \left(\sum_{i=j_0}^{j}\beta_i\right)^{m-1-\left\lfloor\frac{k}{2}\right\rfloor}
        \!\!\!\!\le
        \left(\sum_{i=j_0}^{j+1}\beta_i\right)^{\mathrlap{m-\left\lfloor\frac{k}{2}\right\rfloor}}
        \mathrlap{.}^{\phantom{m-\left\lfloor\frac{k}{2}\right\rfloor}}
    \end{multline}
    By induction, the conditions~\eqref{eq:U_bound} hold for all $j\ge j_0$. Therefore,
    by~\eqref{eq:beta_less_1_2},
    \[
        u^{(j)}_{k,2m}+u^{(j)}_{k,2m+1}
        \le
        \big({\textstyle\sum_{i=j_0}^\infty}\beta_i\big)^{m-\left\lfloor\frac{k}{2}\right\rfloor}
        \le 2^{-m+\left\lfloor\frac{k}{2}\right\rfloor},
    \]
    where $m=1,2,\dots$ and $k=1,2,\dots,2m$. As a consequence,
    \[
    \|U_j\|_1 = \sup_{1\le k<\infty}\left(u^{(j)}_{k,1}+\sum_{m=1}^\infty\left(
            u^{(j)}_{k,2m} + u^{(j)}_{k,2m+1}\right)\right)
    \le\sum_{m=0}^\infty2^{-m}=2.
    \]
    Since $u^{(j+1)}_{k,1} - u^{(j)}_{k,1}=0$ and
    $u^{(j+1)}_{k,2m-1} - u^{(j)}_{k,2m-1}=u^{(j)}_{k,2m-2}$, $m>1$, we have
    \begin{multline}\nonumber
        \|U_{j+1}-U_j\|_1 = \sup_{1\le k<\infty}\sum_{m=1}^\infty\left(
            \left|u^{(j+1)}_{k,2m-1} - u^{(j)}_{k,2m-1}\right|+
            \left|u^{(j+1)}_{k,2m} - u^{(j)}_{k,2m}\right|\right)\le\\
        \sup_{1\le k<\infty}\left(\sum_{m=2}^\infty u^{(j)}_{k,2m-2} +
            \sum_{m=1}^\infty u^{(j)}_{k,2m} +
            \sum_{m=1}^\infty u^{(j+1)}_{k,2m}
        \right) \le\\
        \sup_{1\le k<\infty}\left(2\sum_{m=1}^\infty u^{(j)}_{k,2m} +
            \sum_{m=1}^\infty u^{(j+1)}_{k,2m}
        \right) \le\\
        \sup_{1\le k<\infty}\Bigg(2 \beta_j \sum_{m=1}^\infty u^{(j-1)}_{k,2m-1} +
            \beta_{j+1} \sum_{m=1}^\infty u^{(j)}_{k,2m-1}
        \Bigg)\le\\
        2 \beta_j \|U_{j-1}\|_1 + \beta_{j+1} \|U_{j}\|_1
        \le
        4\beta_j+2\beta_{j+1} \xrightarrow{j\to\infty} 0.
    \end{multline}
    That is, $(U_j)_{j=j_0}^{\raisebox{0.3\height}{$\scriptstyle\infty$}}$ is a Cauchy sequence,
    hence it converges to its entry-wise limit~$U_*$. So we have shown that $\|U_j\|_1$ is
    bounded uniformly in $j$ and there exists $U_*$ satisfying $\|U_j-U_*\|_1\xrightarrow{j\to\infty}0$.
    
    Using~\eqref{eq:norm_1rho} we obtain
    $\|V\|_1 \le \|J(b_0,\beta_0)\|_1\|J(1,\beta_1)\|_1\cdots\|J(1,\beta_{j_0-1})\|_1<\infty$ and
    \begin{multline}\label{eq:VUHT_vanish}
        \big\|V\,U_j\,H_{j+1}\,T(g) - V\,U_*\,H(1,1)\,T(g)\big\|_\rho\le\\
        \big\|V\|_1\,\|U_j\,H_{j+1} - U_*\,H_{j+1} +U_*\,H_{j+1} - U_*\,H(1,1)\big\|_1\,\|T(g)\|_\rho\le\\
        \big\|V\big\|_1\big\|U_j - U_*\big\|_1
        \big\|H_{j+1}\big\|_1\big\|T(g)\big\|_\rho
        + \big\|V\big\|_1\big\|U_*\big\|_1
        \big\|H_{j+1}-H(1,1)\big\|_1\big\|T(g)\big\|_\rho.
    \end{multline}
    The expansion of~$g(z)$ into a power series at the origin converges for $|z|\le\rho$
    absolutely, hence~$\|T(g)\|_\rho<\infty$. According to Corollary~\ref{cr:norm_H},
    \[
    \|H_{j+1}\|_1 = \|H(p_{j+1},p_j)\|_1 < M\text{\quad and\quad}
    \|H(p_{j+1},p_j)-H(1,1)\|_1\xrightarrow{j\to\infty} 0,
    \]
    so the right-hand side of~\eqref{eq:VUHT_vanish} vanishes and, therefore,
    \[
    H(p,q)=VU_*H(1,1)T(g),\quad\text{where}\quad
    VU_*=\lim_{j\to\infty} \left(J(b_0,\beta_0)\,J(1,\beta_1) \cdots J(1,\beta_j)\right).
    \]
\end{proof}
\begin{remark}
    Applying this lemma to the ratio $\frac{p_{j_0-1}(z)}{p_{j_0}(z)}$ we can explicitly determine the
    matrix~$U_*$. Since
    \[u^{(j+1)}_{k,2m}=\beta_{j+1}u^{(j)}_{k,2m-1}\xrightarrow{j\to\infty}0\text{ for
    }m=1,2,\dots\text{ and }k=2m,2m+1,\dots,\]
    from $H_{j_0}=U_*H(1,1)$ we get
    \[U_* = H_{j_0}\,H^{\mathsf{T}}(0,1),\]
    where $A^{\mathsf{T}}$ stands for the transpose of a matrix $A$.
\end{remark}

Now consider meromorphic functions $p(z)\eqqcolon p_0(z)$ and
$q(z)\eqqcolon p_{-1}(z)\not\equiv{0}$, $q(0)\ge 0$, with the power series expansions given
by~\eqref{eq:pq_exp}. Suppose that~$\beta_0$, $\beta_1$, \dots, $\beta_{j-1}\ne 0$. Then we can define
the functions~$p_1(z)$, $p_2(z)$, \dots,~$p_j(z)$ via the formul\ae~\eqref{eq:process1}.

\begin{lemma}\label{lemma:PNNSP}
    If the Hurwitz-type matrix $H_j\coloneqq H(p_j,p_{j-1})$ satisfies the conditions
    \begin{equation}\label{eq:PartPos}
        H_j\pmat 2&3&\hdots&k-1&k\\2&3&\hdots&k-1&i\epmat \ge 0
        \text{, where } k=2,3,4,\dots\text{ and } i = k, k+1, k+2,\dots,
    \end{equation}
    then
    \[p_{j-1}(0)\,p_j(z)\equiv p_{j-1}(z) \iff \beta_j \coloneqq H_j\pmat 2&3\\2&3\epmat=0.\]
\end{lemma}
\begin{proof}
    Without loss of generality we assume $j=0$. Suppose that~$b_0\,p_0\equiv p_{-1}$. Then the
    minor~$\beta_0=H_0\pmat 2&3\\2&3\epmat$ has two proportional rows, and is therefore zero.

    The converse we prove by contradiction. Let $\beta_0 = b_1-a_1b_0=0$ and
    $p_0(z) \not\equiv b_0\,p_{-1}(z)$. Then there exists an integer $i>1$ such that
    $b_i\ne b_0a_i$ (since $p_{-1}(z)\not\equiv 0$). Therefore, according to~\eqref{eq:PartPos}
    we have
    \begin{equation*}
        H_0\pmat 2&3\\2&i+1\epmat
        = \begin{vmatrix}a_0&a_i\\b_0&b_i\end{vmatrix}>0.
    \end{equation*}
    Consequently,
    \[
    H_0\pmat 2&3&4\\2&3&i+2\epmat
    = \begin{vmatrix}a_0&a_1&a_i\\b_0&b_1&b_i\\0&a_0&a_{i-1}\end{vmatrix}
    = - a_0\begin{vmatrix}a_0&a_i\\b_0&b_i\end{vmatrix}<0,
    \]
    which contradicts the conditions~\eqref{eq:PartPos}.
\end{proof}

\begin{corollary}\label{cr:TNN-decomp-hurw-matr}
    Let meromorphic functions $p(z)$ and $q(z)$ be such that $p(0)=1$ and $q(0)\ge 0$. If the
    matrix~$H(p,q)$ satisfies the conditions~\eqref{eq:PartPos}, then
    $F\coloneqq\frac {q}{p}\in\mathcal S$.
\end{corollary}
\begin{proof}
    For $q(z)\equiv 0$ this corollary is obvious. Suppose that $q(z)\not\equiv 0$. Set
    $\relpenalty=10000 p_0(z)\coloneqq p(z)$ and $p_{-1}(z)\coloneqq q(z)$ such that $F(z)=F_0(z)$.

    Suppose that for some $j\ge 0$ we have constructed the sequences $p_{-1}(z),\dots,p_j(z)$ and
    $\beta_0,\dots,\beta_{j-1}>0$. By Lemma~\ref{lemma:H_lpm} the matrix~$H_{j}$
    satisfies~\eqref{eq:PartPos}. Therefore, according to Lemma~\ref{lemma:PNNSP}, we have two
    mutually exclusive possibilities: $\beta_j>0$ and $p_{j-1}(0)\cdot p_j(z)=p_{j-1}(z)$. Consider the
    latter case. We have~$H_j^{(2)} = p_{j}(0)=1$ and $H_j^{(3)}=H_j^{(4)}=\dots=0$ with the
    notation~\eqref{eq:H_minor_sm}. Additionally, the numbers $\beta_i$ are positive for all
    $i=0,1,\dots,j-1$. By virtue of Lemma~\ref{lemma:H_lpm} the minors~$H_0^{(k)}$ are positive
    for~$k=2,\dots,j+2$, and zero for $k>j+2$. So by Theorem~\ref{th:cheb_mei_pos} the
    considered function~$F(z)$ is a rational $\mathcal{S}$-function.
    
    If $\beta_j>0$ we can define the function~$p_{j+1}(z)$ by~\eqref{eq:process1}. According to
    Lemma~\ref{lemma:H_lpm} the matrix~$H_{j+1}$ satisfies~\eqref{eq:PartPos}. So we can make
    the next step of this algorithm. If this process is infinite, by Lemma~\ref{lemma:H_lpm} all
    the principal minors~$H_0^{(k)}$ are positive.  Hence~$F\in\mathcal S$.
\end{proof}
{
    \renewcommand{\thetheorem}{\ref{th:main1}}
\begin{theorem}
    Consider the ratio~$F(z)=\frac{q(z)}{p(z)}$ of power
    series~$p(z)=\sum_{k=0}^\infty\,a_kz^k$ and $q(z)=\sum_{k=0}^\infty\,b_kz^k$, normalized by
    the equality~$p(0)=a_0=1$. The following conditions are equivalent:
    \begin{enumerate}
    \item\label{item:1} The infinite Hurwitz-type matrix
        \begin{equation}\label{eq:Hpq_def2}
            H(p,q)={\begin{pmatrix}
                    b_0& b_1& b_2 & b_3 & b_4 & b_5 & \hdots\\
                    0 & a_0 & a_1 & a_2 & a_3 & a_4 & \hdots\\
                    0 & b_0 & b_1 & b_2 & b_3 & b_4 & \hdots\\
                    0 &   0 & a_0 & a_1 & a_2 & a_3 & \hdots\\
                    0 &   0 & b_0 & b_1 & b_2 & b_3 & \hdots\\
                    \vdots & \vdots & \vdots & \vdots & \vdots & \vdots & \ddots
                \end{pmatrix}}
        \end{equation}
        is totally nonnegative.
    \item \label{item:2} The matrix $H(p,q)$ possesses the infinite factorization
        \begin{equation}\label{eq:H_inf_fact2}
            H(p,q)= \lim_{j\to\infty} \big( J(b_0,\beta_0)\,J(1,\beta_1) \cdots
            J(1,\beta_j)\big) \,H(1,1)\,T(g),
        \end{equation}
        converging in~$\|\cdot\|_\rho$-norm for some $\rho$, $0<\rho\le 1$. Here $b_0\ge 0$ and the
        sequence $(\beta_j)_{j\ge{0}}$ is nonnegative, has a finite sum and contains no zeros
        followed by a nonzero entry, that is
        \[
        \begin{gathered}
            \beta_0,\beta_1,\dots,\beta_{\omega-1}>0,
            \quad\beta_{\omega}=\beta_{\omega+1}=\dots=0,\\
            0\le\omega\le\infty,
            \quad\text{and}
            \quad\sum_{j=0}^\infty\beta_j<\infty.
        \end{gathered}
        \]
        The matrix~$T(g)$ denotes a totally nonnegative Toeplitz matrix of the
        form~\eqref{eq:T_form} with ones on its main diagonal.
    \item \label{item:3} The ratio~$F(z)$ is a meromorphic $\mathcal S$-function; its numerator $q(z)$ and
        denominator $p(z)$ are entire functions of genus~$0$ up to a common meromorphic factor~$g(z)$
        of the form~\eqref{eq:tnns_generator}, $g(0)=1$.
    \end{enumerate}
\end{theorem}
\addtocounter{theorem}{-1} }
\begin{proof}
    We are proving as
    follows:~\ref{item:3}$\implies$\ref{item:2}$\implies$\ref{item:1}$\implies$\ref{item:3}.
    
    The implication~\ref{item:3}$\implies$\ref{item:2} is a consequence of
    Lemma~\ref{lemma:dec_general} since total positivity of the matrix $T(g)$ is provided by
    Theorem~\ref{th:AESW}.

    The factors in~\eqref{eq:H_inf_fact2} are totally nonnegative and have at most a finite
    number of nonzero entries in each column. Therefore, the Cauchy-Binet formula is valid and
    the products of the form
    \[
    J(b_0,\beta_0)\,J(1,\beta_1) \cdots J(1,\beta_{\omega})\,H(1,1)\,T(g),
    \text{ where }0\le\omega<\infty,
    \]
    are totally nonnegative. Moreover, the minors depend continuously on the matrix entries, so
    the entry-wise limit of totally nonnegative matrices is totally nonnegative itself. As a
    consequence, the matrix $H(p,q)$ is totally nonnegative if the condition~\ref{item:2} holds.
    That is the implication~\ref{item:2}$\implies$\ref{item:1} is true.

    Let us show that~\ref{item:1}$\implies$\ref{item:3}. The matrices~$T(p)$ and~$T(q)$ are
    submatrices of $H(p,q)$. Thus, if~\ref{item:1} is true, by Theorem~\ref{th:AESW} the series
    $p(z)$ and $q(z)$ converge to meromorphic functions of the form~\eqref{eq:tnns_generator},
    so $F(z)$ is a meromorphic function as well.
    
    Corollary~\ref{cr:TNN-decomp-hurw-matr} yields $F\in\mathcal{S}$. So according to
    Lemma~\ref{lemma:cheb1}, the zeros (we denote their number by~$\omega_1\le\infty$) and poles (we
    denote their number by~$\omega_2\le\infty$) of $F(z)$ are real, simple and interlacing. Moreover,
    $F(x)>F(0)\ge 0$ for $x>0$, hence $F(z)$ has only the zeros $-\tau_1$, $-\tau_2$, \dots,
    $-\tau_{\omega_1}$ and poles $-\sigma_1$, $-\sigma_2$, \dots, $-\sigma_{\omega_2}$,
    satisfying the following condition ({\em cf.}~\eqref{eq:pq_zeros})
    \begin{equation}\label{eq:order1}
        0\le\tau_1<\sigma_1<\tau_2<\sigma_2<\tau_3<\dots.
    \end{equation}
    Note that~\eqref{eq:order1} implies the inequality $\omega_1-1\le\omega_2\le\omega_1$.

    However, the functions~$p(z)$ and $q(z)$ have the form~\eqref{eq:tnns_generator}, in particular
    they have no nonpositive poles. Therefore, all the numbers $-\tau_1$, $-\tau_2$, \dots,
    $-\tau_{\omega_1}$ are among the zeros of $q(z)$, while $-\sigma_1$, $-\sigma_2$, \dots,
    $-\sigma_{\omega_2}$ are among the zeros of $p(z)$. As a result,
    \[
    q(z)=e^{\gamma_1 z}\tilde{q}(z)\frac{\prod_{\nu} \left(1+\frac {z}{\alpha_\nu}\right)}
    {\prod_{\mu} \left(1-\frac {z}{\beta_\mu}\right)}
    \quad\text{and}\quad
    p(z)=e^{\gamma_2 z}\tilde{p}(z)\frac{\prod_{\nu} \left(1+\frac {z}{\alpha_\nu}\right)}
    {\prod_{\mu} \left(1-\frac {z}{\beta_\mu}\right)},
    \]
    where $\gamma_1$, $\gamma_2$, $(\alpha_\nu)_\nu$ and $(\beta_\mu)_\mu$ are appropriately
    chosen positive numbers,
    \[
    \tilde{q}(z)\coloneqq
    b_0z^j\prod_{\nu=j+1}^{\omega_1}\left(1+\frac{z}{\tau_\nu}\right)\text{ for } j=1-\sign b_0
    \text{ and }
    \tilde{p}(z)\coloneqq\prod_{\mu=1}^{\omega_2}\left(1+\frac{z}{\sigma_\mu}\right).
    \]
    After cancellations in the fraction $\dfrac{q(z)}{p(z)}$ we obtain
    $F(z)=e^{\gamma z}\dfrac{\tilde{q}(z)}{\tilde{p}(z)}$, where
    $\relpenalty=10000 \binoppenalty=10000 \gamma\coloneqq\gamma_1-\gamma_2$.

    Let us show $\gamma=0$. Set
    \[G(z)\coloneqq\tfrac{1}{b_0}e^{-\gamma z}F(z).\]
    For $\omega_1,\omega_2<\infty$ we can express the rational function~$G(z)$ as a sum of
    partial fractions and ascertain that it agrees with the expansion~\eqref{F_sum_form_R}. So
    in this case~$G\in\mathcal{S}$.

    Suppose that $\omega_1$ and $\omega_2$ are infinite. With
    $\Arg:\mathbb{C}\setminus\{0\}\to(-\pi,\pi]$ denoting the principal argument, we obtain the
    following for $\Im z>0$ from~\eqref{eq:order1}
    \[\pi
    > \Arg \left(\tau_0+z\right)
    > \Arg \left(\sigma_0+z\right)
    > \Arg \left(\tau_1+z\right)
    > \Arg \left(\sigma_1+z\right)
    > \dots> 0.\]    
    Whenever $\nu$ or $\tau_0$ is positive, we obtain
    \begin{equation}\label{eq:aux1}
        0 < \Arg\left(1+\frac{z}{\tau_\nu}\right) - \Arg\left(1+\frac{z}{\sigma_{\nu}}\right)
        < \Arg\left(1+\frac{z}{\tau_{\nu}}\right) - \Arg\left(1+\frac{z}{\tau_{\nu+1}}\right)<\pi.
    \end{equation}
    Therefore, if $\tau_0>0$
    \begin{multline}\label{eq:aux2}
        0<\Arg\prod_{\nu=0}^\infty \frac{1+\frac{z}{\tau_\nu}}{1+\frac{z}{\sigma_\nu}}\le
        \sum_{\nu=0}^\infty \Arg \frac{1+\frac{z}{\tau_\nu}}{1+\frac{z}{\sigma_\nu}} =\\
        \sum_{\nu=0}^\infty \left(\Arg\left({1+\frac{z}{\tau_\nu}}\right)
            -\Arg\left({1+\frac{z}{\sigma_\nu}}\right)\right) <\\
        \Arg\left(1+\frac{z}{\tau_0}\right) -
        \lim_{\nu\to\infty}\Arg\left(1+\frac{z}{\tau_{\nu}}\right) <\pi.
    \end{multline}
    That is
    \begin{equation}\label{eq:aux3}
        0<\Arg G(z)<\pi\quad\text{when}\quad \Im z>0,
    \end{equation}
    \emph{i.e.} $G\in\mathcal{S}$ since $G(z)$ is real. If $\tau_0=0$ we just replace all
    instances of~$\left(1+\frac{z}{\tau_{0}}\right)$ in inequalities~\eqref{eq:aux1}
    and~\eqref{eq:aux2} with $z$ and obtain the same. This method to deduce the
    estimate~\eqref{eq:aux3} is taken from~\cite[Ch.IV~\S 10, Lemma~11]{ChebMei}.

    Now if $\gamma\ne 0$, then $\Arg F({\pi}i/{\gamma}) = \Arg(-G({\pi}i/{\gamma}))$ and
    $G\in\mathcal{S}$, which contradicts the inclusion~$F\in\mathcal{S}$. Thus $\gamma=0$.
\end{proof}

\section{Distribution of zeros and poles}\label{sec:appl-entire-funct}
First, let us prove the following auxiliary fact.
\begin{lemma}\label{lemma:H_tnn-p_zero}
    Let $p(z)$, $q(z)$ be the formal power series $p(z)=\sum_{k=0}^\infty\,a_kz^k$ and
    $q(z)=\sum_{k=0}^\infty\,b_kz^k$ such that $a_0=0$ and $b_0>0$. The matrix $H(p,q)$ defined
    by~\eqref{eq:Hpq_def2} is totally nonnegative if and only if $p(z)\equiv 0$ and $q(z)$
    converges to a function of the form
    \begin{equation}\label{eq:tnns_generator2}
        f(z)=f_0\,e^{\gamma z}\frac{\prod_{\nu} \left(1+\frac {z}{\alpha_\nu}\right)}
        {\prod_{\mu} \left(1-\frac {z}{\beta_\mu}\right)},
    \end{equation}
    where
    $\gamma\ge 0$, $\alpha_\nu,\beta_\mu>0$ for all $\mu,\nu$ and
    $\sum_{\nu}\frac{1}{\alpha_\nu}+\sum_{\mu}\frac{1}{\beta_\mu} <\infty$.
\end{lemma}
\begin{proof}
    If $a_0=0$ then the total nonnegativity of~$H=H(p,q)$ implies
    \[
    0\le a_{i}
    = -\frac{1}{b_0}\begin{vmatrix}a_0&a_i\\b_0&b_i\end{vmatrix}
    = -\frac{1}{b_0} H \pmat 2&3\\2&i+1\epmat \le 0\quad
    \forall i=1,2,3,\dots,
    \]
    so $p(z)\equiv 0$. The Toeplitz matrix~$T(q)$ defined by~\eqref{eq:T_form} is totally
    nonnegative as a submatrix of~$H(p,q)$. Therefore, by Theorem~\ref{th:AESW}, $q(z)$ is of
    the form~\eqref{eq:tnns_generator2}.

    Conversely, if~$p(z)\equiv 0$ then any nonzero minor of $H(p,q)$ is equal to a minor of
    $T(q)$. According to Theorem~\ref{th:AESW} the matrix~$T(q)$ is totally nonnegative, hence
    $H(p,q)$ is totally nonnegative as well.
\end{proof}
{
    \renewcommand{\thetheorem}{\ref{th:negroots_p0dec}}
\begin{theorem}
    A power series~$f(z)={\sum_{k=0}^\infty}\, f_kz^k$ with $f_0> 0$ converges to an entire function of
    the form
    \begin{equation}\label{eq:ent_tnns_generator2}
        f(z)=f_0\,e^{\gamma z}\prod_{1\le\nu\le\omega} \left(1+\frac {z}{\alpha_\nu}\right),
    \end{equation} 
    where $\gamma\ge 0$, $\alpha_\nu>0$ for all $\nu$ and
    $\sum_{1\le\nu\le\omega}\frac{1}{\alpha_\nu} <\infty$ for some $\omega$,
    $0\le\omega\le\infty$,
    if and only if the infinite matrix
    \[
    \mathcal D_f={\begin{pmatrix}
        f_0 & f_1 & f_2 & f_3 & f_4 & \hdots\\
        0 & f_1 & 2f_2 & 3f_3 & 4f_4 & \hdots\\
        0 & f_0 & f_1 & f_2 & f_3 & \hdots\\
        0 &     0 & f_1 & 2f_2& 3f_3 & \hdots\\
        0 &     0 & f_0 & f_1 & f_2 & \hdots\\
        \vdots & \vdots & \vdots & \vdots & \vdots & \ddots
    \end{pmatrix}}
    \]
    is totally nonnegative.
\end{theorem}
\addtocounter{theorem}{-1} }

\begin{proof}
    Let the matrix~$\mathcal D_f=H(f',f)$ be totally nonnegative, where~$f'(z)$ denotes the
    formal derivative of $f(z)$. If $f_1=0$, by Lemma~\ref{lemma:H_tnn-p_zero}
    $f(z)\equiv f_0>0$, \emph{i.e.} \eqref{eq:ent_tnns_generator2} is satisfied.

    Suppose that~$f_1\ne 0$. Theorem~\ref{th:main1} implies that $f(z)$ and $f'(z)$ converge in
    a neighbourhood of the origin. Moreover, for some meromorphic function~$g(z)$ of the
    form~\eqref{eq:tnns_generator2}, the functions $\tilde{f}(z)\coloneqq\frac{f(z)}{g(z)}$ and
    $h\coloneqq\frac{f'(z)}{g(z)}$ are entire of genus~$0$, coprime and have only negative
    zeros. In particular, the poles of~$g(z)$ are positive (if any). Let us show~$g(z)$ has no
    poles. Observe that in the right-hand side of the expression
    \[
    h(z)=\frac{\tilde{f}(z)g'(z)+g(z)\tilde{f}'(z)}{g(z)}
        = \tilde{f}'(z) + \tilde{f}(z)\frac{g'(z)}{g(z)}
    \]
    the logarithmic derivative of $g(z)$ is multiplied by a function with no positive
    zeros. Therefore, each pole of $g(z)$ must be a pole of $h(z)$. But $h(z)$ is an entire
    function, thus~$g(z)$ is entire and $f(z)=\tilde{f}(z)g(z)$ can be represented as
    in~\eqref{eq:ent_tnns_generator2}.
    
    Conversely, let~$f(z)$ admit the representation~\eqref{eq:ent_tnns_generator2}. If $f(z)$ is
    a constant then by Lemma~\ref{lemma:H_tnn-p_zero} the matrix~$\mathcal D_f$ is totally
    nonnegative. Suppose now that $f(z)$ is not a constant and consider its logarithmic
    derivative
    \begin{equation}\label{eq:log_deriv}
        F(z)=\frac{f'(z)}{f(z)}=\gamma+\sum_{1\le\nu\le\omega}\frac 1{z+\alpha_\nu},\qquad
        0\le\omega\le\infty.
    \end{equation}
    Each summand in the right-hand side of~\eqref{eq:log_deriv} is in~$\mathcal{R}^{-1}$,
    so~$F\in\mathcal R^{-1}$. The function~$f(z)$ is non-constant, hence $F(z)\not\equiv 0$.
    Therefore, the function $\frac 1{F(z)}$ is an $\mathcal{R}$-function. Moreover, $F(z)$ (and
    hence~$\frac 1{F(z)}$) has only negative poles and zeros. Consequently, $\frac 1{F(z)}$ is
    an $\mathcal S$-function.
    
    Since $f(z)$ has the form~\eqref{eq:ent_tnns_generator2}, each common zero of $f(z)$ and
    $f'(z)$ is negative. In addition the functions $e^{-\gamma z}f(z)$ and $e^{-\gamma z}f'(z)$
    are of genus~$0$. So the function $\frac{f(z)}{f'(z)}=\frac 1{F(z)}$ satisfies~\ref{item:3}
    in Theorem~\ref{th:main1}. Consequently, the matrix
    \[\mathcal D_f=H(f',f)=f_1\,H\left(\tfrac{f'}{f_1},\tfrac{f}{f_1}\right)\]
    is totally nonnegative.
\end{proof}

Let $f(z)=\sum_{k=0}^\infty f_kz^k$, $f_0>0$, be a real entire function. Define its infinite
Hurwitz matrix
    \begin{equation}\label{mcH_form}
        \mathcal H_f={\begin{pmatrix}
                f_0 & f_2 & f_4 & f_6 & f_8 & \hdots\\
                0 & f_1 & f_3 & f_5 & f_7 & \hdots\\
                0 & f_0 & f_2 & f_4 & f_6 & \hdots\\
                0 &   0 & f_1 & f_3 & f_5 & \hdots\\
                0 &   0 & f_0 & f_2 & f_4 & \hdots\\
                \vdots & \vdots & \vdots & \vdots & \vdots & \ddots
            \end{pmatrix}}.
    \end{equation}
    For the minors of $\mathcal H_f$ we use the same notation as in
    Section~\ref{sec:basic_facts}, such that
\[
\mathcal H_f^{(k)} = \mathcal H_f{\pmat 2&3&\dots& k\\2&3&\dots& k \epmat}.
\]

Grommer in~\cite[\S 16, Satz~IV]{Grommer} extended the Hurwitz criterion~\cite{Hurwitz} to
entire functions. However, he overlooked the condition on common zeros of odd and even parts
(which was addressed by Kre\u\i n in~\cite{AkhKrein}). We only need the following particular
case of this extension.
\begin{theorem}[{\cite[Theorem~12]{AkhKrein}, \cite[Ch.V~\S 4]{ChebMei}}]\label{th:stable_p01dec}
    Let a real entire function~$f(z)$, $f(0)>0$, be of genus~$1$ or~$0$. Suppose that its even part
    $\big(f(z)+f(-z)\big)/2$ and its odd part $\big(f(z)-f(-z)\big)/2$ have no common zeros.

    Then the function $f$ can be represented as
    \begin{equation}\label{Hrw_func}
        f(z) = C e^{\gamma z}\prod_{1\le\mu\le\omega_1} \left(1+\frac{z}{x_\mu}\right)
        \prod_{1\le\nu\le\omega_2} \left(1+\frac{z}{\alpha_\nu}\right)
        \left(1+\frac{z}{\overline \alpha_\nu}\right),
    \end{equation}
    where $0\le\omega_1,\omega_2\le\infty$, $\gamma\ge 0$ and $C>0$, and its zeros satisfy the
    conditions
    \begin{gather}
        x_\mu > 0\text{ for }1\le\mu\le\omega_1,\quad
        \Re\alpha_\nu > 0,\ \Im\alpha_\nu > 0\text{ for }1\le\nu\le\omega_2,
        \nonumber\\
        \sum_{1\le\mu\le\omega_1} \frac 1{x_\mu}<\infty\quad\text{and}
        \quad\sum_{1\le\nu\le\omega_2}\Re\frac{1}{\alpha_\nu}<\infty,\nonumber
    \end{gather}
    if and only if
    \[
    \begin{gathered}
        \mathcal H_f^{(2)}(f), \mathcal H_f^{(3)}(f), \dots, \mathcal H_f^{(\omega_1+2\omega_2+1)}(f)>0,\\
        \mathcal H_f^{(\omega_1+2\omega_2+2)}(f)=\mathcal H_f^{(\omega_1+2\omega_2+3)}(f)=\dots=0.
    \end{gathered}
    \]
\end{theorem}
Note that the restriction on the genus of $f(z)$ implies the additional condition
$\sum_{1\le\nu\le\omega_2}\frac 1{|\alpha_\nu|^2}<\infty$. Based on Theorems~\ref{th:main1}
and~\ref{th:stable_p01dec} we deduce the following fact.
{
    \renewcommand{\thetheorem}{\ref{th:qstable_tnn}}
\begin{theorem}
    Given a power series $f(z)=z^j\sum_{k=0}^\infty\,f_kz^k$, where $f_0>0$ and $j$ is a
    nonnegative integer, the infinite matrix $\mathcal H_f$ defined by~\eqref{mcH_form} is
    totally nonnegative if and only if the series~$f(z)$ converges to a function of the form
    \begin{equation} \label{eq:f_as_quo2}
        f(z)=Cz^je^{\gamma_1 z+\gamma_2 z^2}
        \frac{
            \prod_{1\le\mu\le\omega_1}\left( 1+\frac {z}{x_\mu}\right)\prod_{1\le\nu\le\omega_2}\left(
                1+\frac {z}{\alpha_\nu}\right)\left(1+\frac {z}{\overline\alpha_\nu}\right)}{
        \prod_{1\le\lambda\le\omega_3}\left( 1+\frac {z}{y_\lambda}\right)
        \left( 1-\frac {z}{y_\lambda}\right)},
    \end{equation}
    for some~$\omega_1$, $\omega_2$ and~$\omega_3$, $0\le\omega_1,\omega_2,\omega_3\le\infty$.
    Here $C>0$,
    \begin{gather}\label{eq:p_form3}
        \gamma_1, \gamma_2\ge 0,\  x_\nu,y_\lambda > 0,\ \Re\alpha_\nu\ge 0\text{ and }
        \Im\alpha_\nu> 0\text{ for all }\mu, \nu, \lambda,
        \\ \label{eq:p_form4}
        \sum_{1\le\mu\le\omega_1} \frac 1{x_\mu} +
        \sum_{1\le\nu\le\omega_2}\Re\left(\frac{1}{\alpha_\nu}\right) +
        \sum_{1\le\nu\le\omega_2}\frac 1{|\alpha_\nu|^2} + \sum_{1\le\lambda\le\omega_3}
        \frac 1{y_\lambda^2}<\infty.
    \end{gather}
\end{theorem}
    \addtocounter{theorem}{-1}
}

\begin{proof}
    Suppose that $f(z)$ is represented as~\eqref{eq:f_as_quo2}. We can express it as
    \[
    f(z)=Cz^{j}g(z^2)h(z),\quad \text{where}
    \]
    \begin{gather}
        h(z)\coloneqq e^{\gamma_1 z} \prod_{1\le\mu\le\omega_1}\left( 1+\frac {z}{x_\mu}\right)
        \prod_{\substack{\Re\alpha_\nu>0\\1\le\nu\le\omega_2}}\left(
            1+\frac {z}{\alpha_\nu}\right)\left(1+\frac{z}{\overline\alpha_\nu}\right)
        \quad\text{and}\label{eq:h_form}\\
        g(z^2)\coloneqq e^{\gamma_2 z^2}
        \prod_{\substack{\Re\alpha_\nu=0\\1\le\nu\le\omega_2}}\left(
            1+\frac {z^2}{i^2\alpha_\nu^2}\right)
        \Big/ \prod_{1\le\lambda\le\omega_3}\left( 1-\frac {z^2}{y_\lambda^2}\right).\label{eq:g_form}
    \end{gather}
    Note that $g(z)$ has the form~\eqref{eq:tnns_generator2}, so by Theorem~\ref{th:AESW} its
    Toeplitz matrix $T(g)$ is totally nonnegative.
    
    Split $h(z)$ into the odd part $zh_o(z^2)$ and the even part $h_e(z^2)$ so that
    \[
    h(z) = h_e(z^2) + z h_o(z^2).
    \]
    The function $h(z)$ is of genus not exceeding~$1$ as well as $h_e(z^2)$ and $h_o(z^2)$. This
    implies that the genus of $h_o(z)$ and $h_e(z)$ is~$0$. Indeed, for example, the
    function~$h_e(z^2)$ with zeros~$\pm\delta_n$, $n=1,2,\dots$, can be represented as the
    Weierstra\ss\ product
    \[
    h_e(z^2)
    =e^{cz}\prod_n\left(1-\frac{z}{\delta_n}\right)e^{\dfrac{z}{\delta_n}}
      \left(1+\frac{z}{\delta_n}\right)e^{-\dfrac{z}{\delta_n}}
    =e^{cz}\prod_n\left(1-\frac{z^2}{\delta_n^2}\right).
    \]
    And since it depends only on $z^2$, we necessarily have $c=0$ in this representation, which
    implies
    \[
    h_e(z) =\prod_n\left(1-\frac{z}{\delta_n^2}\right).
    \]

    Let us show $h_o(z)$ and $h_e(z)$ are coprime. Denote $r\coloneqq\gcd(h_o,h_e)$ if
    $h_o(z)\not\equiv 0$. If $h_o(z)\equiv 0$ we set $r(z)\coloneqq h_e(z)$.  Assume that
    $r(z)\not\equiv 1$. So it has zeros, since its genus is zero and $r(0)=h_e(0)=1$. However,
    if $r(z_0)=0$ then $h(\sqrt{z_0})=h(-\sqrt{z_0})=0$. Since one of the points
    $\pm\sqrt{z_0}\ne 0$ is in the closed right half of the complex plane (independently of the
    branch of the square root) we get a contradiction to~\eqref{eq:h_form}.

    If $h_o(z)\equiv 0$ then $H(h_o,h_e)=H(0,1)$ is totally nonnegative. This implies the total
    nonnegativity of the matrix
    \[\mathcal H_f=C\,H(gh_o,gh_e) =C\,H(h_o,h_e)\,T(g).\]
    If $h_o(z)\not\equiv 0$, then the function~$h(z)$ has the form~\eqref{eq:h_form} and its odd
    and even parts are coprime. That is, $h(z)$ satisfies the conditions of
    Theorem~\ref{th:stable_p01dec}. Therefore, $\mathcal H_{h}^{(2)}=H^{(2)}(h_o,h_e)$,
    $\mathcal H_{h}^{(3)}=H^{(3)}(h_o,h_e)$,\dots\ is a positive sequence possibly followed by
    zeros. Thus, by Theorem~\ref{th:cheb_mei_pos}, we obtain $\frac{h_e}{h_o}\in\mathcal S$.
    Then the applying of Theorem~\ref{th:main1} to the function $\frac{h_e}{h_o}$ gives the
    total nonnegativity of the matrices $H(h_o,h_e)$ and, consequently,~$\mathcal H_f$.
    
    Let us prove the converse. Suppose that the Hurwitz matrix $\mathcal H_f$ is totally
    nonnegative. We can split the series
    $f_0^{-1}z^{-j}f(z)=\sum_{k=0}^\infty\,\frac{f_k}{f_0}z^k$ into the even part $q(z^2)$ and
    the odd part $z p(z^2)$ so that $f(z)$ can be expressed as follows
    \[
        f(z) = f_0 z^j \left(q(z^2) + z p(z^2)\right).
    \]
    It gives~$\mathcal H_f=f_0H(p,q)$.

    If $f_1>0$, then since the matrix~$H(p,q)$ is totally nonnegative, by Theorem~\ref{th:main1}
    there exists a meromorphic function~$g(z)$ of the form~\eqref{eq:tnns_generator2} such that
    $\tilde{p}(z)=\frac{p(z)}{g(z)}$ and $\tilde{q}(z)=\frac{q(z)}{g(z)}$ are coprime entire
    functions of genus~$0$. Moreover, the ratio~$\frac{\tilde{q}(z)}{\tilde{p}(z)}$ is an
    $\mathcal{S}$-function. Let
    \[\tilde{f}(z)\coloneqq\tilde{q}(z^2) + z \tilde{p}(z^2).\]
    By Theorem~\ref{th:cheb_mei_pos} the minors
    $\mathcal H_{\tilde{f}}^{(2)}=H^{(2)}(\tilde{p},\tilde{q})$,
    $\mathcal H_{\tilde{f}}^{(3)}=H^{(3)}(\tilde{p},\tilde{q})$,\dots\ form a positive
    sequence possibly followed by zeros. Since $\tilde{q}(z)$ and $\tilde{p}(z)$ are coprime,
    the function~$\tilde{f}(z)$ has the form~\eqref{Hrw_func} by Theorem~\ref{th:stable_p01dec}.

    If $f_1=0$ then according to Lemma~\ref{lemma:H_tnn-p_zero}, $p(z)\equiv 0$ and $q(z)$ has
    the form~\eqref{eq:tnns_generator2}. Here we set $g(z)\coloneqq q(z)$ so that
    $\tilde{f}(z)\equiv 1$.

    Now consider both cases $f_1=0$ and $f_1>0$. We showed that~$\tilde{f}(z)$ can be
    represented as in~\eqref{Hrw_func}, while~$g(z)$ can be represented as
    in~\eqref{eq:tnns_generator2}. That is, after the appropriate renaming of zeros and poles,
    the function~$\tilde{f}(z)$ has the form~\eqref{eq:h_form}, the function~$g(z^2)$ has the
    form~\eqref{eq:g_form} and the conditions~\eqref{eq:p_form3}--\eqref{eq:p_form4} are
    satisfied. Thereby
    \[
    f(z) = f_0 z^j g(z^2)\tilde{f}(z)
    \]
    can be represented as in~\eqref{eq:f_as_quo2}.
\end{proof}

\addtocontents{toc}{\SkipTocEntry}
\section*{Acknowledgements}
The author is grateful to Olga Holtz and Mikhail Tyaglov for helpful comments
    and stimulating discussions.

\end{document}